\documentclass{amsart}

\usepackage{amssymb,latexsym,amsfonts,amsmath}
\usepackage{graphicx,epsfig}
\include{diagrams}
\usepackage{graphicx}
\usepackage{color}
\usepackage{diagrams}
\usepackage{amsfonts}
\usepackage{amssymb,latexsym,amsfonts,amsmath}
\usepackage{graphicx}
\usepackage{tikz}
\usepackage{subfigure}
\usetikzlibrary{automata}

\usepackage{savesym}
\savesymbol{AND}
\usepackage{algorithmic}
\usepackage{algorithm}

\newtheorem{theorem}{Theorem}[section]
\newtheorem{lemma}[theorem]{Lemma}

\newtheorem{proposition}[theorem]{Proposition}
\newtheorem{corollary}[theorem]{Corollary}

\theoremstyle{definition}
\newtheorem{definition}[theorem]{Definition}
\newtheorem{example}[theorem]{Example}

\theoremstyle{remark}
\newtheorem{remark}[theorem]{Remark}

\numberwithin{equation}{section}

\newcommand{\temp}{\mathrm{temp}}

\newcommand{\Obs}{\mathrm{Obs}}
\newcommand{\CRP}{\mathrm{CRP}}
\newcommand{\GalR}{\mathrm{GalR}}
\newcommand{\GalP}{\mathrm{GalP}}
\newcommand{\MglB}{\mathrm{MglB}}
\newcommand{\cAMP}{\mathrm{cAMP}}
\newcommand{\Dgal}{\mathrm{D-gal}}

\begin{document}

\begin{abstract}                          
Motivated by safety--critical applications in cyber--physical systems, in this paper we study the notion of critical observability and design of observers for networks of Finite State Machines (FSMs). Critical observability is a property of FSMs that corresponds to the possibility of detecting if the current state of an FSM belongs to a set of critical states modeling operations that may be unsafe or, in general, operations of specific interest in a particular application. A critical observer is an observer that detects on--line the occurrence of critical states. When a large--scale network of FSMs is considered, the construction of such an observer is prohibitive because of the large computational effort needed. In this paper we propose a decentralized architecture for critical observers of networks of FSMs, where on--line detection of critical states is performed by local critical observers, each associated with an FSM of the network. For the efficient design of decentralized critical observers we first extend on--the--fly algorithms traditionally used in the community of formal methods for the formal verification and control design of FSMs. We then extend to networks of FSMs, bisimulation theory traditionally given in the community of formal methods for single FSMs. The proposed techniques provide a remarkable computational complexity reduction, as discussed throughout the paper and also demonstrated by means of illustrative examples. Finally, we propose an example in the context of biological networks, which illustrates the applicability of our results and also the interest arising from concrete application domains.
\end{abstract}

\title[Decentralized Critical Observers of Networks of Finite State Machines and Model Reduction]{Decentralized Critical Observers of Networks of Finite State Machines and Model Reduction}

\thanks{The research leading to these results has been partially supported by the Center of Excellence DEWS.}

\author[Giordano Pola, Davide Pezzuti, Elena De Santis and Maria D. Di Benedetto]{
Giordano Pola, Davide Pezzuti, Elena De Santis and Maria D. Di Benedetto
}
\address{$^1$Department of Information Engineering, Computer Science and Mathematics, Center of Excellence DEWS,
University of L{'}Aquila, 67100 L{'}Aquila, Italy}
\email{
\{giordano.pola,davide.pezzuti,elena.desantis,mariadomenica.dibenedetto\}@univaq.it}

\maketitle

\section{Introduction}
In recent years, Cyber--Physical Systems have been intensively
investigated by both academic and industrial communities because they offer
a solid paradigm for the modeling, analysis and control of many next
generation large--scale, complex, distributed and networked engineered
systems. Among them, safety--critical applications, as for example Air
Traffic Management (ATM) systems, play a prominent role, see e.g. \cite{MAREA,ThesisPezzuti}. Ensuring safety in
large--scale and networked safety--critical applications is a tough but
challenging problem. In particular, complexity is one of the most difficult
issues that must be overcome to make theoretical methodologies applicable to
real industrial applications. \newline
In this paper we address the analysis of critical observability and
design of observers for networks of Finite State Machines (FSMs). A network of
FSMs is a collection of FSMs whose interaction is captured by the notion of
parallel composition. Critical observability is a property that corresponds
to the possibility of detecting if the current state of an FSM belongs to 
a set of critical states, 
modeling operations that may be unsafe or, in general, operations of specific interest in a particular application. 
This notion has been introduced in \cite{DeSantis06} for linear switching
systems and is relevant in safety--critical applications where the timely recovery of human errors and device failures is of
primary importance in ensuring safety. Current approaches available in the literature to
check critical observability are based on regular language theory as in \cite{IJRNC08} or on the design of the so--called critical observers \cite{Cassandras,DeSantis06}. The computational complexity of the first
approach is polynomial in the number of states of the FSM, while the one of the
second is exponential. Although disadvantageous from the computational
complexity point of view, the construction of critical observers cannot be
avoided at the implementation layer since it is necessary for the automatic on--line
detection of critical situations. Motivated by this issue we elaborated on some results which can reduce, in
some cases drastically, the computational effort in designing critical
observers for large--scale networks of FSMs. 
We first propose a decentralized architecture for critical observers of the network, which is
composed of a collection of local critical observers, each associated
with an FSM of the network.
Efficient algorithms for the synthesis of critical observers are proposed, which extend on--the--fly techniques traditionally used by the community of formal methods for verification and control of FSMs (see e.g. \cite{onthefly3,onthefly2}). We then propose results on model reduction which extend to networks of FSMs, bisimulation theory \cite{Milner,Park} traditionally given by the community of formal methods for single FSMs.        
We define a bisimulation equivalence that takes into account criticalities. We then reduce the original
network of FSMs to a smaller one, obtained as the quotient of the original
network induced by the bisimulation equivalence. In the reduction process,
FSMs composing the network are never composed, a key factor in complexity reduction. We first show that
critical observability of the original network implies and is implied by the
critical observability of the quotient network. We then show that a decentralized critical
observer for the original network can be easily derived from the one
designed for the quotient network. Finally, we propose an example in the context of biological networks, which illustrates the applicability of our results and also the interest arising from concrete applications domain (see also \cite{ThesisPezzuti} for ATM applications). \\
While our approach that extends some techniques from formal methods (see e.g. \cite{ModelChecking}) 
is important \textit{per se} at a theoretical level, it also provides a set of new methodologies for 
computational complexity reduction in designing observers. Critical observability belongs to the set of observability\footnote{Here, the term "observability" is used in a broader sense.} concepts extensively studied in the literature of discrete event systems.
To the best of our knowledge, the formal methods techniques proposed in this paper have not yet been explored in the community of discrete event systems neither for the analysis of critical observability nor for the analysis of any other observability notion, with the only exception of \cite{Wonham:03}. We defer to the last section a discussion on connections with the existing literature.
\newline
This paper is organized as follows.
Section \ref{sec2} introduces notation, networks of FSMs and the
critical observability property. Section \ref{sec3bis} presents decentralized critical observers and Section \ref{sec4} model reduction via bisimulation equivalence. 
An example on the analysis of a biological network is reported in Section \ref{sec6}. In Section \ref{sec7} we provide a discussion on connections with existing literature and outlook.

\section{Networks of Finite State Machines and critical observability}

\label{sec2} 
In this section, we start by introducing our notation in
Subsection \ref{sec2.1}. We then recall the notions of networks of finite
states machines in Subsection \ref{sec2.2} and of critical observability in Subsection \ref{sec2.3}.

\subsection{Notation and preliminary definitions}

\label{sec2.1} The symbols $\wedge $ and $\vee $ denote the \emph{And} and 
\emph{Or} logical operators, respectively. The symbol $\mathbb{N}$ denotes
the set of nonnegative integers. Given $n,m\in \mathbb{N}$ with $n<m$ let be 
$[n;m]=[n,m]\cap \mathbb{N}$. The symbol $|X|$ denotes the cardinality of a
finite set $X$. The symbol $2^{X}$ denotes the power set of a set $X$. Given
a function $f:X\rightarrow Y$ we denote by $f(Z)$ the image of a set $
Z\subseteq X$ through $f$, i.e. $f(Z)=\{y\in Y|\exists z\in Z\text{ s.t. }
y=f(z)\}$; if $X^{\prime }\subset X$ and $Y^{\prime }\subset Y$ then $
f|_{X^{\prime }\rightarrow Y^{\prime }}$ is the restriction of $f$ to domain 
$X^{\prime }$ and co--domain $Y^{\prime }$, i.e. $f|_{X^{\prime }\rightarrow
Y^{\prime }}(x)=f(x)$ for any $x\in X^{\prime }$ with $f(x)\in Y^{\prime }$.
We now recall from \cite{Cassandras} some basic notions of language theory.
Given a set $\Sigma $, a finite sequence $w=\sigma _{1}\sigma _{2}\sigma _{3}...$
with symbols $\sigma _{i}\in \Sigma $ is called a word in $\Sigma $; the
empty word is denoted by $\varepsilon $. 
The Kleene closure $w^{\ast }$ of a word $w$ is the collection of
words $\varepsilon $, $w$, $ww$, $www$, \mbox{... .} The symbol $\Sigma
^{\ast }$ denotes the set of all words in $\Sigma $, including the
empty word $\varepsilon $. The concatenation of two words $u,v\in \Sigma
^{\ast }$ is denoted by $uv\in \Sigma ^{\ast }$. Any subset of $\Sigma
^{\ast }$ is called a language. The projection of a language $L\subseteq
\Sigma ^{\ast }$ onto a subset $\widehat{\Sigma }$ of $\Sigma $ is the
language $P_{\widehat{\Sigma }}(L)=\{t\in {\widehat{\Sigma }}^{\ast
}|\exists w\in L\text{ s.t. }P_{\widehat{\Sigma }}(w)=t\}$ where $P_{
\widehat{\Sigma }}(w)$ is inductively defined for any $w\in L$ and $\sigma
\in \Sigma $ by $P_{\widehat{\Sigma }}(\varepsilon )=\varepsilon $ and $P_{
\widehat{\Sigma }}(w\sigma )=P_{\widehat{\Sigma }}(w)\sigma $ if $\sigma \in 
\widehat{\Sigma }$ and $P_{\widehat{\Sigma }}(w\sigma )=P_{\widehat{\Sigma }
}(w)$, otherwise. 

\subsection{Networks of Finite State Machines}

\label{sec2.2} 
In this paper we consider the class of nondeterministic FSMs with observable labels:

\begin{definition}
\label{fsm} A Finite State Machine (FSM) $M$ is a tuple $(X,X^0,\Sigma,
\delta)$ where $X$ is the set of states, $X^0\subseteq X$ is the set of
initial states, $\Sigma$ is the set of input labels and $\delta:X \times
\Sigma \rightarrow 2^X$ is the transition map. 
\end{definition}

A state run $r$ of an FSM $M$ is a sequence $x^{0}\rTo^{\sigma ^{1}}x^{1}\rTo
^{\sigma ^{2}}$ $x^{2}\rTo^{\sigma ^{3}}x^{3}\,...$ such that $x^{0}\in X^{0}
$, $x^{i}\in X$, $\sigma ^{i}\in \Sigma $ and $x^{i+1}\in \delta
(x^{i},\sigma ^{i+1})$ for any $x^{i}$ and $\sigma ^{i}$ in the sequence;
the sequence $\sigma ^{1}\,\sigma ^{2}\,\sigma ^{3}\,...$ is called the
trace associated with $r$. For $X^{\prime }\subseteq X$ and $\sigma \in
\Sigma $, we abuse notation by writing $\delta (X^{\prime },\sigma )$
instead of $\bigcup_{x\in X^{\prime }}\delta (x,\sigma )$. The extended
transition map $\hat{\delta}$ associated with $\delta $ is inductively
defined for any $w\in \Sigma ^{\ast }$, $\sigma \in \Sigma $ and $x\in X$ by 
$\hat{\delta}(x,\varepsilon )=\{x\}$ and $\hat{\delta}(x,w\sigma
)=\bigcup_{y\in \hat{\delta}(x,w)}\delta (y,\sigma )$. 
The language generated by $M$, denoted $L(M)$, is composed by all traces
generated by $M$, or equivalently, $L(M)=\{w\in \Sigma ^{\ast }|\exists
x^{0}\in X^{0}\text{ s.t. }\hat{\delta}(x^{0},w)\neq \varnothing \}$. 
An FSM $M$ is deterministic if $|X^{0}|=1$ and $|\delta (x,\sigma )|\leq 1$, for any $x\in X$ and $\sigma
\in \Sigma $. 
In this paper we are interested in studying whether it is possible to detect
if the current state of an FSM $M$ is or is not in a set of critical states $C\subset X$ modeling operations that may be unsafe or, in general, operations of specific interest in a particular application. We refer to an FSM $
(X,X^{0},\Sigma ,\delta )$ equipped with a set of critical states $C$ by the
tuple $(X,X^{0},\Sigma ,\delta ,C)$. We also refer to an FSM with outputs by
a tuple $(X,X^{0},\Sigma ,\delta ,Y,H)$, where $Y$ is the set of output
labels and $H:X\rightarrow Y$ is the output function. 
For simplicity we call an FSM equipped with critical states or with outputs
as an FSM. 
The operator $\mathrm{Ac}(\cdot )$ extracts the accessible part from an FSM $M=(X,X^{0},\Sigma ,\delta ,C)$ (resp. $M=(X,X^{0},\Sigma ,\delta ,Y,H)$),
i.e. $\mathrm{Ac}(M)=(X',X^0,\Sigma,\delta^{\prime},C^{\prime})$
(resp. $\mathrm{Ac}(M)=(X',X^0,\Sigma,\delta^{\prime},Y,H^{\prime})
$) where $X'=\{x\in X | \exists x^0 \in X^{0}\wedge w\in \Sigma ^{\ast }\text{ s.t. }x\in 
\hat{\delta}(x^{0},w)\}$, $\delta ^{\prime }=\delta |_{X^{\prime }\times
\Sigma \rightarrow X^{\prime }}$, $C^{\prime }=C\cap X^{\prime }$ and $
H^{\prime }=H|_{X^{\prime }\rightarrow Y}$. 
Interaction among FSMs is captured by the following notion of composition.

\begin{definition}
\label{composition} The parallel composition $M_1||M_2=
\left(X_{1,2},X_{1,2}^{0},\Sigma_{1,2},\delta_{1,2},C_{1,2}\right) $ between
two FSMs $M_1=(X_1,X_{1}^{0},\Sigma_1,\delta_1,C_1)$ and $
M_2=(X_2,X_{2}^{0},\Sigma_2,\delta_2,C_2)$ is the FSM $\mathrm{Ac}
(X^{\prime }_{1,2},X_{1,2}^{\prime,0},\Sigma^{\prime }_{1,2},\delta^{\prime
}_{1,2},C^{\prime }_{1,2}) $ where $X^{\prime }_{1,2}=X_1 \times X_2$, $
X_{1,2}^{\prime,0}=X_{1}^{0} \times X_{2}^{0}$, $\Sigma^{\prime
}_{1,2}=\Sigma_1 \cup \Sigma_2$, $C^{\prime }_{1,2}=\left(C_1 \times
X_2\right) \cup \left(X_1 \times C_2\right)$ and $\delta^{\prime
}_{1,2}:X^{\prime }_{1,2} \times \Sigma^{\prime }_{1,2} \rightarrow
2^{X^{\prime }_{1,2}}$ is defined for any $x_1 \in X^{\prime }_1$, $x_2 \in
X^{\prime }_2$ and $\sigma \in \Sigma^{\prime }_{1,2}$ by \newline
\begin{equation*}
\begin{cases}
\delta_1(x_1,\sigma) \times \delta_2(x_2,\sigma), \text{if }
\delta_1(x_1,\sigma) \neq \varnothing \wedge \,\,\delta_2(x_2,\sigma) \neq
\varnothing \\ 
\quad \quad \quad \quad \quad \quad \quad \quad \quad \wedge \,\, \sigma \in
\Sigma_1 \cap \Sigma_2 , \\ 
\delta_1(x_1,\sigma) \times \{x_2\}, \text{if } \delta_1(x_1,\sigma) \neq
\varnothing \wedge \,\,\sigma \in \Sigma_1 \backslash \Sigma_2 , \\ 
\{x_1\} \times \delta_2(x_2,\sigma), \text{if } \delta_2(x_2,\sigma) \neq
\varnothing \wedge \,\,\sigma \in \Sigma_2 \backslash \Sigma_1, \\ 
\varnothing, \text{otherwise.}
\end{cases}
\end{equation*}
\end{definition}

By definition, a state $(x_{1},x_{2})\in C_{1,2}$, i.e.  $(x_{1},x_{2})$ is
considered as critical for $M_{1}||M_{2}$, if and only if $x_{1}\in C_{1}$ or $
x_{2}\in C_{2}$. Vice versa, $(x_{1},x_{2})\notin C_{1,2}$ if and only if $
x_{1}\notin C_{1}$ and $x_{2}\notin C_{2}$. 
It is well known that

\begin{proposition}
\label{Com+Ass-prop} \cite{Cassandras} The parallel composition operation is
commutative up to isomorphisms and associative.
\end{proposition}

By the above property of parallel composition, we may write in the sequel $
M_{1}||M_{2}||M_{3}$, $X_{1,2,3}$ and $C_{1,2,3}$ instead of $
M_{1}||(M_{2}||M_{3})$, $X_{1,(2,3)}$ and $C_{1,(2,3)}$ or, instead of $
(M_{1}||M_{2})||M_{3}$, $X_{(1,2),3}$ and $C_{(1,2),3}$. \\
In this paper we consider a network 
\begin{equation*}
\mathcal{N}=\{M_{1},M_{2},...,M_{N}\}
\end{equation*}
of $N$ FSMs $M_{i}$ whose interaction is captured by the notion of parallel
composition; the corresponding FSM is given by $\mathbf{M}(\mathcal{N})=M_{1}||M_{2}||...||M_{N}$. 
The FSM $\mathbf{M}(\mathcal{N})$ is well
defined because the composition operator $||$ is associative. The definition
of parallel composition among an arbitrary number of FSMs is reported in
e.g. \cite{Cassandras} and coincides with the recursive application of the
binary operator $||$, as in Definition \ref{composition}. 
For the computational complexity analysis, we will use in the sequel the
number $n_{\max }=\max_{i\in \lbrack 1;N]}|X_{i}|$ as indicator of the sizes
of the FSMs composing the network $\mathcal{N}$. An upper bound to space and
time computational complexity in constructing $\mathbf{M}(\mathcal{N})$ is $
O(2^{N\log (n_{\max })})$.

\begin{figure}[tbp]
\begin{center}
\includegraphics[scale=0.24]{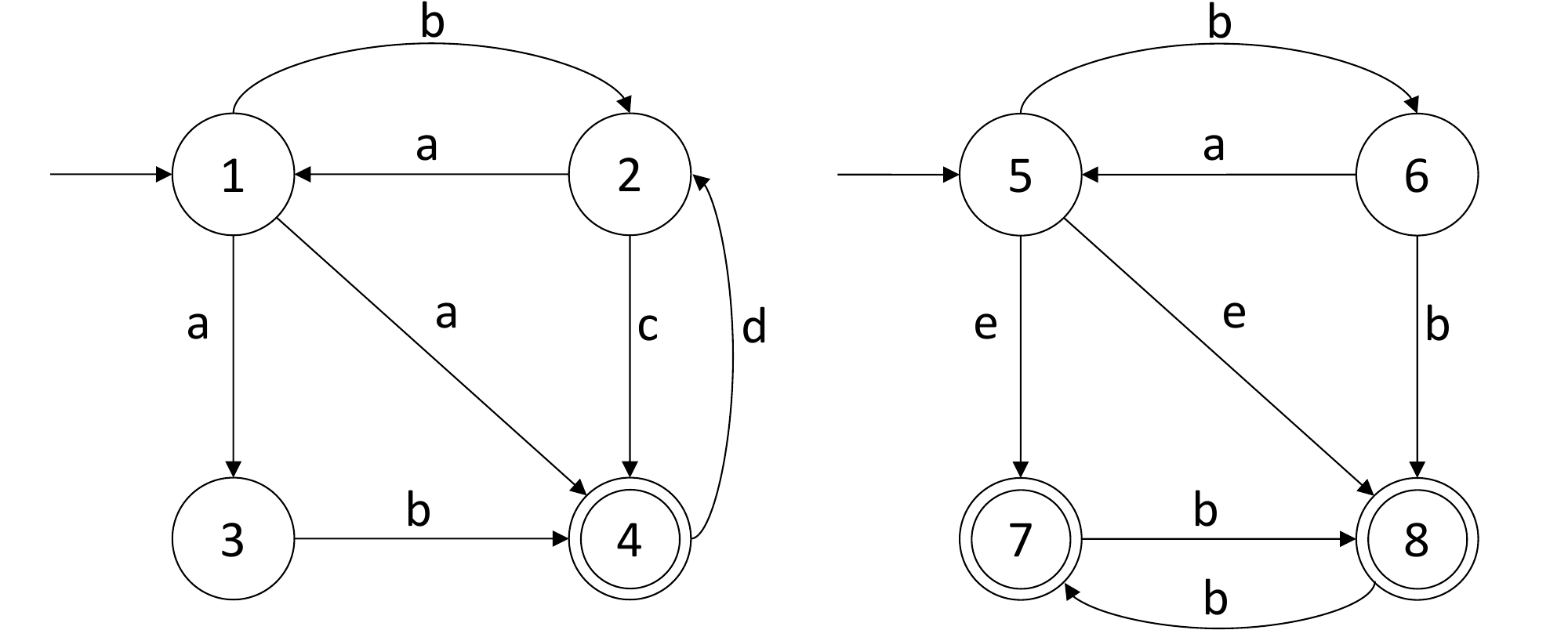}
\end{center}
\caption{FSM $M_1$ in the left and FSM $M_2$ in the right.}
\label{fig1}
\end{figure}

\subsection{Critical observability and observers}
\label{sec2.3} 

Critical observability corresponds to the possibility of
detecting whether the current state $x$ of a run of an FSM is or
is not critical on the basis of the information given by the corresponding
trace at state $x$:

\begin{definition}
\label{c_obs} An FSM $M=(X,X^{0},\Sigma,\delta,C)$ is critically observable
if $[\hat{\delta}(x^0,w) \subseteq C] \vee [\hat{\delta}(x^0,w) \subseteq
X\backslash C], $ for any initial state $x^0 \in X^{0}$ and any trace $w \in
L(M)$. 
\end{definition}

Any FSM $M$ having an initial state that is critical and another initial
state that is not critical, is never critically observable. Moreover, if $X^0=C$ then FSM $M$ is critically observable and no further analysis for the detection of critical states is needed. For these reasons
in the sequel we assume that $[\,X^{0}\subset C\,]\,\vee
\,[\,X^{0}\subseteq X\backslash C\,]$ for any FSM $M$. An illustrative
example follows.

\begin{example}
\label{example1} Consider FSMs $M_{i}=(X_{i},X_{i}^{0},\Sigma _{i},\delta
_{i},C_{i})$, $i=1,2$, depicted in Fig. \ref{fig1}, where $X_{1}=\{1,2,3,4\}$
, $X_{1}^{0}=\{1\}$, $\Sigma _{1}=\{a,b,c,d\}$, $C_{1}=\{4\}$, $
X_{2}=\{5,6,7,8\}$, $X_{2}^{0}=\{5\}$, $\Sigma _{2}=\{a,b,e\}$, $
C_{2}=\{7,8\}$ and transition maps $\delta _{1}$ and $\delta _{2}$ are
represented by labeled arrows in Fig. \ref{fig1}; labels on the arrows
represent the input label associated with the corresponding transition. FSM $
M_{1}$ is not critically observable because it is possible to reach both
noncritical state $3$ and critical state $4$ starting from the initial state 
$1$, by applying the same input label $a$. FSM $M_{2}$ is critically
observable because by applying traces $b(ab)^{\ast }$ and $b(ab)^{\ast }a$
to the initial state $5$, the state reached is always in $X_{2}\backslash
C_{2}$ while by applying any trace other than the previous ones, states
reached are always critical.
\end{example}

\begin{remark}
\label{remark1}
The notion of critical observability has been inspired by the notion of current location observability, e.g. \cite{BalluchiHSCC02}. We recall that an FSM $M$ is current location observable if there exists an integer $K$ such that for any run $x^{0}\rTo^{\sigma ^{1}}x^{1}\rTo
^{\sigma ^{2}}$ $x^{2}\rTo^{\sigma^{3}}x^{3}\rTo^{\sigma^{4}}\,...\rTo^{\sigma ^{k}}x^{k}$ of $M$ with length $k\geq K$, the current state $x^k$ of $M$ can be reconstructed on the basis of the sequence $\sigma^1 \sigma^2\, ... \, \sigma^k$. Current state observability is also termed strong detectability in e.g. \cite{Detectability:07}. 
While in current state observability the issue is to detect the current state from some step on, in critical observability the issue is to detect if the current state is or is not in a set of critical states. Moreover, we now show that critical observability does not imply and is not implied by current state observability. 
Consider the FSM $M_1=(X_1,X^{0}_1,\Sigma_1,\delta_1,C_1)$ where $X_1=\{0,1,2\}$, $X^{0}_1=\{0\}$, $\Sigma_1=\{a\}$, 
$\delta_1(0,a)=\{1,2\}$, $\delta_1(1,a)=\{2\}$, $\delta_1(2,a)=\{1\}$ and $C_1=\{1,2\}$: FSM $M_1$ is critically observable (starting from state $0$ with label $a$ only critical states $1$ and $2$ are reached) and is not current state observable (from the sequence $aa^\ast$ it is not possible to detect if the current state is $1$ or $2$). Conversely, consider the FSM $M_2=(X_2,X^{0}_2,\Sigma_2,\delta_2,C_2)$ where $X_2=\{0,1,2,3\}$, $X^{0}_2=\{0\}$, $\Sigma_2=\{a\}$, 
$\delta_2(0,a)=\{1,2\}$, $\delta_2(1,a)=\delta_2(2,a)=\{3\}$, $\delta_2(3,a)=\{3\}$ and $C_2=\{1\}$: FSM $M$ is not critically observable (starting from state $0$ with label $a$ both critical state $1$ and noncritical state $2$ are reached) and is current state observable with $K=2$ (after the second transition the current state will be $3$ for ever). 
\end{remark}

On--line detection of critical states of critically observable FSMs can be
obtained by means of critical observers, as defined hereafter.

\begin{definition}
\label{defobsabstract} A deterministic FSM $\mathrm{Obs}=(X_\mathrm{Obs},$ $
X^{0}_\mathrm{Obs},\Sigma_\mathrm{Obs},\delta_\mathrm{Obs},Y_\mathrm{Obs},H_
\mathrm{Obs}) $ with output set $Y_\mathrm{Obs}=\{0,1\}$ is a critical
observer for an FSM $M=(X,X^{0},\Sigma,\delta$, $C)$ if $\Sigma_\mathrm{Obs}
=\Sigma$ and for any state run $r: x^0 \rTo^{\sigma^1} x^1$ $\rTo^{\sigma^2}
x^2 \rTo^{\sigma^3} x^3 \, ... $ of $M$, the corresponding (unique) state
run $r_\mathrm{Obs}: z^0 \rTo^{\sigma_1} z^1 \rTo^{\sigma_2} z^2 \rTo
^{\sigma_3} z^3 \, ... $ of $\mathrm{Obs}$ is such that $H_\mathrm{Obs}
(z^i)=1$ if $x^i\in C$ and $H_\mathrm{Obs}(z^i)=0$ otherwise, 
for any state $z^i$ in $r_\mathrm{Obs}$.
\end{definition}

For later use, we report from e.g. \cite{Cassandras} the following
construction of observers.

\begin{definition}
\label{defobs} Given an FSM $M=(X,X^{0},\Sigma ,\delta ,C)$, define the
deterministic FSM $
\mathrm{Obs}(M)$ 
as the accessible part $\mathrm{Ac}(\mathrm{Obs}^{\prime })$ of $\mathrm{Obs}
^{\prime }=(X_{\mathrm{Obs}^{\prime }},X_{\mathrm{Obs}^{\prime }}^{0}$, $
\Sigma _{\mathrm{Obs}^{\prime }},\delta _{\mathrm{Obs}^{\prime }},$ $Y_{
\mathrm{Obs}^{\prime }},H_{\mathrm{Obs}^{\prime }})$, where $X_{\mathrm{Obs}
^{\prime }}=2^{X}$, $X_{\mathrm{Obs}^{\prime }}^{0}=\{X^{0}\}$, $\Sigma _{
\mathrm{Obs}^{\prime }}=\Sigma $, $\delta _{\mathrm{Obs}^{\prime }}:X_{
\mathrm{Obs}^{\prime }}\times \Sigma _{\mathrm{Obs}^{\prime }}\rightarrow
2^{X_{\mathrm{Obs}}^{\prime }}$ is defined by $\delta _{\mathrm{Obs}^{\prime
}}(z,\sigma )=\{\bigcup_{x\in z}\delta (x,\sigma )\}$, $Y_{\mathrm{Obs}
^{\prime }}=\{0,1\}$ and $H_{\mathrm{Obs}^{\prime }}(z)=1$ if $z\cap C\neq
\varnothing $ and $H_{\mathrm{Obs}^{\prime }}(z)=0$, otherwise. 
\end{definition}

An upper bound to the space and time computational complexity in
constructing $\mathrm{Obs}(M)$ for an FSM $M$ with $n=|X|$ states is $
O(2^{n})$ from which, an upper bound to the space and time computational
complexity in constructing $\mathrm{Obs}(\mathbf{M}(\mathcal{N}))$ is $
O(2^{2^{N\log (n_{\max })}})$. A direct consequence of Definitions \ref
{c_obs}, \ref{defobsabstract} and \ref{defobs} is the following:

\begin{proposition}
\label{thobs} The following statements are equivalent:\newline
(i) $M$ is critically observable;\newline
(ii) $\mathrm{Obs}(M)=(X_\mathrm{Obs},$ $X^{0}_\mathrm{Obs},\Sigma_\mathrm{
Obs},\delta_\mathrm{Obs},Y_\mathrm{Obs},H_\mathrm{Obs})$ is a critical
observer for $M$;\newline
(iii) For any $z\in X_{\mathrm{Obs}}$, if $H_{\mathrm{Obs}}(z)=1$ then $z
\subseteq C$. 
\end{proposition}

\section{Design of decentralized critical observers}
\label{sec3bis} 

In this section, we first show that critical observability
of all the FSMs composing a network ensures the critical observability of
their parallel composition, i.e. of the network itself.  However, a network
of FSMs can be critically observable even though not all the FSMs composing
the network are critically observable. This means that for checking the
critical observability of a network we need to compose the FSMs, and this
may be problematic especially when dealing with large-scale networks where a 
large number of FSMs has to be composed. 

\begin{proposition}
\label{lemma2} If FSMs $M_1$ and $M_2$ are critically observable then FSM $
M_1||M_2$ is critically observable.
\end{proposition}

\begin{proof}
Set $M_i=\left(X_i,X_{i}^{0},
\Sigma_i,\delta_i,C_i\right)$, $i=1,2$. By contradiction, assume that $
M_1||M_2$ is not critically observable. Thus, there exists a pair of state
runs $r_1$ and $r_2$ with initial states $(x_1^0,x_2^0),(\overline{x}_1^0,
\overline{x}_2^0)\in X_{1,2}^{0}$ and common trace $w$ such that $[(x_1,x_2)
\in \hat{\delta}_{1,2}((x_1^0,x_2^0),w)]\wedge [(\overline{x}_1,\overline{x}
_2) \in \hat{\delta}_{1,2}((\overline{x}_1^0,\overline{x}_2^0),w))]$ and $
[(x_1,x_2)\in X_{1,2}\backslash C_{1,2}] \wedge [(\overline{x}_1,\overline{x}
_2)\in C_{1,2}]$. By definition of the projection operator $
P_{\Sigma_i}(\cdot)$ we then get $[ x_1 \in \hat{\delta}_{1}(x_1^0,P_{
\Sigma_1}(w))] \wedge [ \overline{x}_1 \in \hat{\delta}_{1}(\overline{x}
_1^0,P_{\Sigma_1}(w))] \wedge [ x_2 \in \hat{\delta}_{2}(x_2^0,P_{
\Sigma_2}(w))] \wedge [ \overline{x}_2 \in \hat{\delta}_{2}(\overline{x}
_2^0,P_{\Sigma_2}(w))]$. Moreover, by definition of $C_{1,2}$ we then get $
[[x_1 \in X_1 \backslash C_1 ] \wedge [ x_2 \in X_2 \backslash C_2 ]]\wedge
[[\overline{x}_1 \in C_1 ]\vee [ \overline{x}_2 \in C_2]]$. Hence, either $
M_1$ or $M_2$ is not critically observable and a contradiction holds.
\end{proof}

The converse implication is not true in general, as shown in the following example.

\begin{figure}[tbp]
\begin{center}
\includegraphics[scale=0.33]{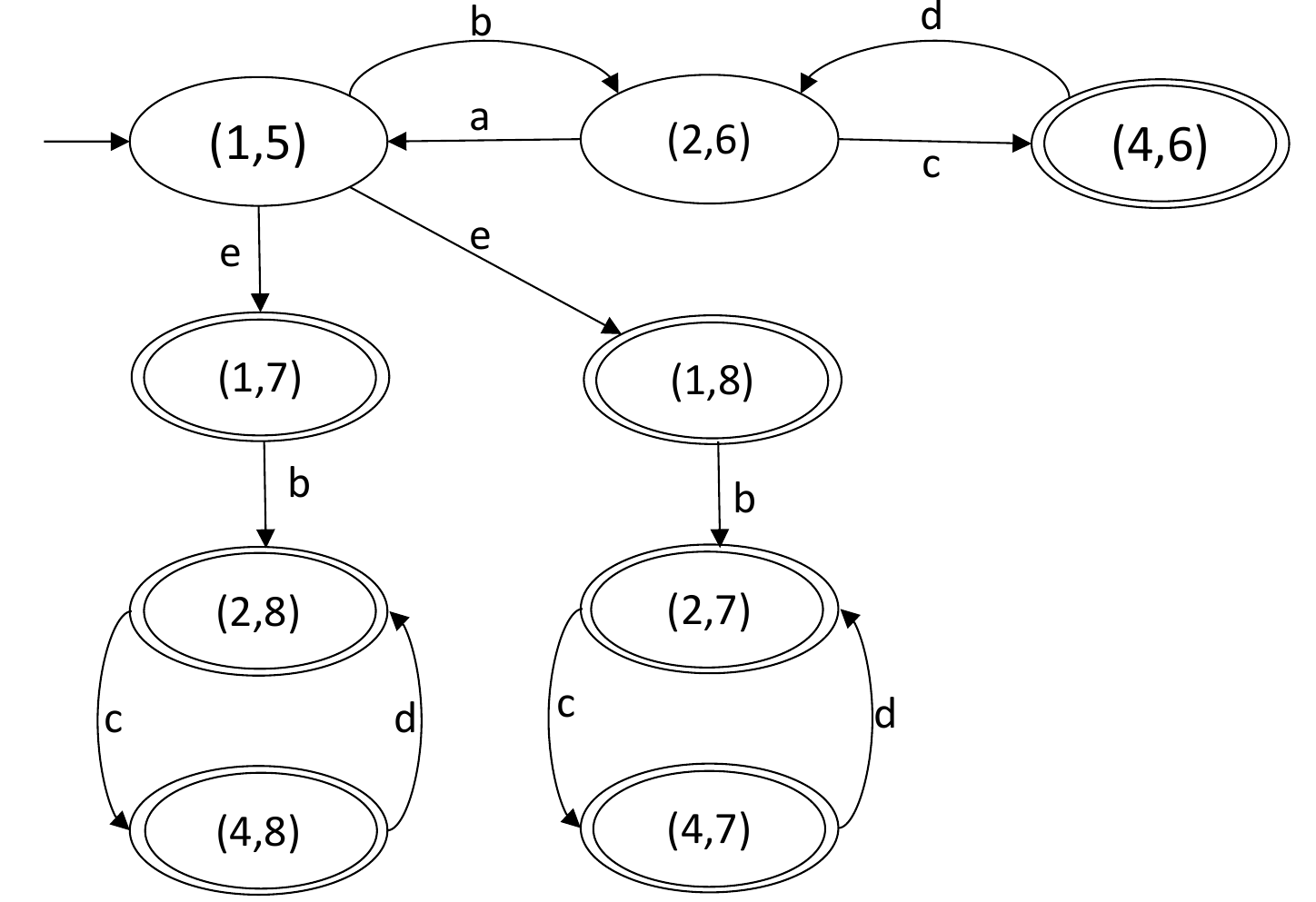}
\end{center}
\caption{Parallel composition $M_1 || M_2$ of FSMs $M_1$ and $M_2$.}
\label{fig3}
\end{figure}

\begin{example}
\label{example2} Consider FSMs $M_1$ and $M_2$ in Example \ref{example1} and
depicted in Fig. \ref{fig1}. As discussed in Example \ref{example1}, FSM $M_2
$ is critically observable while FSM $M_1$ is not. It is readily seen that
FSM $M_1 || M_2$, depicted in Fig. \ref{fig3}, is critically observable
because by applying traces $b((cd)^*ab)^*$ and $b(cd)^*a(b(cd)^*a)^*$ to the
initial state $(1,5)$, the state reached is always in $X_{1,2}\backslash
C_{1,2}$ and by applying any trace other than the previous ones, states
reached are always in $C_{1,2}$.
\end{example}

As a consequence, critical observability of each FSM composing $\mathcal{N}$
is not necessary for $\mathbf{M}(\mathcal{N})$ to be critically observable. 
We now show how decentralized observers can be used for
detecting critical states of the network of FSMs $\mathbf{M}(\mathcal{N})$.
The  notion of isomorphism will be used:

\begin{definition}
\label{DefIso} Two FSMs $M_{i}=(X_{i},X_{i}^0,$ $\Sigma_i,\delta_i,Y_i,H_i),$ $i=1,2$, are isomorphic, denoted $M_1=_{\mathrm{iso}} M_2$, if $
\Sigma_1=\Sigma_2$, $Y_1=Y_2$ and there exists a bijective function $\phi:
X_1 \rightarrow X_2$, such that:\newline
(i) $\phi(X_1^0)=X_2^0$;\newline
(ii) $\phi(X_1)=X_2$;\newline
(iii) $\phi(\delta_{1}(x_{1},\sigma))=\delta_{2}(\phi(x_{1}),\sigma)$, for
any accessible state $x_1\in X_1$ of $M_1$ and $\sigma \in \Sigma_1$;\newline
(iv) $H_1 (x_1) = H_2 (\phi(x_1))$, for any $x_1\in X_1$. 
\end{definition}

\begin{proposition}
\label{isocomp} Given FSMs $M_i$ ($i\in [1;4]$), if $M_1=_{\mathrm{iso}} M_2$
and $M_3=_{\mathrm{iso}} M_4$ then $M_1|| M_3 =_{\mathrm{iso}} M_2|| M_4$.
\end{proposition}

Given $\mathcal{N}$, consider the collection of deterministic FSMs 
\begin{equation}
\mathrm{Obs}(M_{i})=(X_{\mathrm{Obs},i},X_{\mathrm{Obs},i}^{0},\Sigma
_{i},\delta _{\mathrm{Obs},i},Y_{\mathrm{Obs},i},H_{\mathrm{Obs},i}),
\label{ObsMi}
\end{equation}
each associated to the FSM $M_{i}$ and define the decentralized observer $
\mathrm{Obs}^{d}(\mathcal{N})$ as the FSM 
\begin{equation*}
\mathrm{Obs}(M_{1})||\mathrm{Obs}(M_{2})||...||\mathrm{Obs}(M_{N})
\end{equation*}
with output set $Y_{\mathrm{Obs}^{d}}=\{0,1\}$ and output function $H_{
\mathrm{Obs}^{d}}(z_{1},z_{2},...,z_{N})=\bigvee_{i\in \lbrack 1;N]}H_{
\mathrm{Obs}_{i}}(z_{i})$. The following result shows that the decentralized
observer $\mathrm{Obs}^{d}(\mathcal{N})$ can be used for detecting on--line
critical states of the FSM $\mathbf{M}(\mathcal{N})$. 

\begin{theorem}
\label{thObsDistr} $\mathrm{Obs}^d(\mathcal{N})=_{\mathrm{iso}}\mathrm{Obs}(
\mathbf{M}(\mathcal{N}))$. 
\end{theorem}

\begin{proof}
We start by showing the result for the network $\mathcal{N}^{\prime
}=\{M_1,M_2\}$ where $M_i=(X_i,X^{0}_i,\Sigma_i,\delta_i,C_i )$, i.e. 
\begin{equation}  \label{aqw}
\mathrm{Obs}(M_1)||\mathrm{Obs}(M_2)=_{\mathrm{iso}} \mathrm{Obs}(M_1||M_2). 
\end{equation}
Let be $\mathrm{Obs}^{d}(\mathcal{N}^{\prime })=(X_{\mathrm{Obs}^{d}},X^{0}_{
\mathrm{Obs}^{d}},\Sigma_{\mathrm{Obs}^{d}},\delta_{\mathrm{Obs}^{d}},Y_{
\mathrm{Obs}^{d}},$ $H_{\mathrm{Obs}^{d}})$ and $\mathrm{Obs}(\mathbf{M}(
\mathcal{N}^{\prime }))=(X_{\mathrm{Obs}},X^{0}_{\mathrm{Obs}},\Sigma_{
\mathrm{Obs}},\delta_{\mathrm{Obs}},$ $Y_{\mathrm{Obs}},H_{\mathrm{Obs}})$.
Define $\phi:X_{\mathrm{Obs}^{d}} \rightarrow X_{\mathrm{Obs}}$ such that $
\phi((z_1,z_2))=z_1\times z_2$, for any $(z_1,z_2)\in X_{\mathrm{Obs}^{d}}$
with $z_i \in X_{\mathrm{Obs},i}$. First of all note that $\Sigma_{\mathrm{
Obs}^{d}}=\Sigma_{\mathrm{Obs}}=\Sigma_1 \cup \Sigma_2$ and $Y_{\mathrm{Obs}
^{d}}=Y_{\mathrm{Obs}}=\{0,1\}$. Moreover, with reference to Definition \ref
{DefIso}, we get:\newline
\textit{Condition (i).} 
By Definitions \ref{composition} and \ref{defobs} we get $\phi(X^{0}_{
\mathrm{Obs}^{d}}) = \phi(X^{0}_{\mathrm{Obs},1}\times X^{0}_{\mathrm{Obs}
,2}) = \phi(\{X^0_1\}\times \{X^0_2\}) = \phi(\{(X^0_1,X^0_2)\}) =
\{\phi((X^0_1,X^0_2))\} = \{X^0_1 \times X^0_2\} =X^0_{\mathrm{Obs}}$. 
\newline
\textit{Conditions (ii) and (iii).} We proceed by induction and show that if 
$\phi((z_1,z_2))=z_1\times z_2$ for a pair of accessible states $
(z_1,z_2)\in X_{\mathrm{Obs}^d}$ and $z_1 \times z_2 \in X_{\mathrm{Obs}}$
then $\phi(\delta_{\mathrm{Obs}^d}((z_1,z_2),\sigma))=\delta_{\mathrm{Obs}
}(\phi((z_1,z_2)),\sigma)$, for any $\sigma \in \Sigma_1\cup \Sigma_1$. With
reference to Definition \ref{composition}, we have three cases: (case 1) $
\sigma \in \Sigma_1 \cap \Sigma_2 $, (case 2) $\sigma \in \Sigma_1
\backslash \Sigma_2$, and (case 3) $\sigma \in \Sigma_2 \backslash \Sigma_1$. 
We start with case 1. By Definitions \ref{composition} and \ref{defobs} we
get $\delta_{\mathrm{Obs}}(\phi((z_1,z_2)),\sigma) = \delta_{\mathrm{Obs}
}(z_1 \times z_2,\sigma)= \{\bigcup_{(x_1,x_2)\in z_1 \times
z_2}\delta_{1,2}((x_1,x_2),\sigma) \}= \{\bigcup_{(x_1,x_2)\in z_1 \times
z_2}\delta_{1}(x_1,\sigma)\times\delta_{2}(x_2,\sigma) \} =
\{\bigcup_{x_1\in z_1}\delta_{1}(x_1,\sigma)$ $\times \bigcup_{x_2 \in
z_2}\delta_{2}(x_2,\sigma) \}= \{\phi(\bigcup_{x_1\in
z_1}\delta_{1}(x_1,\sigma), \bigcup_{x_2 \in z_2}\delta_{2}(x_2,$ $\sigma))
\}= \phi(\{(\bigcup_{x_1\in z_1}\delta_{1}(x_1,\sigma), \bigcup_{x_2 \in
z_2}\delta_{2}(x_2,\sigma)) \} )= \phi(\{\bigcup_{x_1\in
z_1}\delta_{1}(x_1,\sigma)\} \times \{\bigcup_{x_2 \in
z_2}\delta_{2}(x_2,\sigma)\} )= \phi(\delta_{\mathrm{Obs},1}(z_1,$ $\sigma)
\times \delta_{\mathrm{Obs},2}(z_2,\sigma) ) = \phi(\delta_{\mathrm{Obs}
^{d}}((z_1,z_2),\sigma))$. 
Cases 2 and 3 can be shown by using similar arguments. \newline
\textit{Condition (iv).} Suppose $H_{\mathrm{Obs}^d}(z_1,z_2)=1$. By
definition of $H_{\mathrm{Obs}^d}(.)$ we get that $[H_{\mathrm{Obs}
,1}(z_1)=1] \vee [H_{\mathrm{Obs},2}(z_2)=1]$, or by Definition \ref{defobs}
equivalently that $[z_1 \cap C_1 \neq \varnothing]\vee[z_2 \cap C_2 \neq
\varnothing]$. By definition of the set $C_{1,2}$, the last conditions imply that $z_1 \times z_2 \cap C_{1,2}\neq \varnothing$ which by Definition \ref{defobsabstract}, implies $H_{\mathrm{Obs}}(z_1\times z_2)=1$.
Suppose now $H_{\mathrm{Obs}^d}(z_1,z_2)=0$
. By definition of $H_{\mathrm{Obs}^d}(.)$ we get that $[H_{\mathrm{Obs}
,1}(z_1)=0] \wedge [H_{\mathrm{Obs},2}(z_2)=0]$, or by Definition \ref
{defobs} equivalently that $[z_1 \cap C_1 = \varnothing]\wedge [z_2 \cap C_2
= \varnothing]$. The last conditions imply that $z_1 \times z_2 \cap C_{1,2}
= \varnothing$ which, by Definition \ref{defobs}, implies $H_{\mathrm{Obs}
}(z_1\times z_2)=0$. \newline
Hence, the isomorphic equivalence in (\ref{aqw}) is proven. We now
generalize (\ref{aqw}) to the case of a generic network $\mathcal{N}
=\{M_1,M_2,...,M_N\}$. By applying recursively the equivalence in (\ref{aqw}) and by Proposition \ref{isocomp}, we get 
$\mathrm{Obs}(\mathbf{M}(\mathcal{N})) = \mathrm{Obs}
(M_1||(M_2||M_3||...||M_{N})) =_{\mathrm{iso}} \mathrm{Obs}(M_1)||\mathrm{Obs
}(M_2||(M_3||...||M_{N})) =_{\mathrm{iso}} \mathrm{Obs}(M_1)||$ $\mathrm{Obs}
(M_2)||\mathrm{Obs}(M_3||...||M_{N}) =_{\mathrm{iso}} ... =_{\mathrm{iso}} 
\mathrm{Obs}(M_1)||$ $\mathrm{Obs}(M_2)||...||\mathrm{Obs}(M_{N}) = \mathrm{
Obs}^d(\mathcal{N})$. 
\end{proof}

In the sequel, we will refer to the FSM $\mathrm{Obs}^{d}(\mathcal{N})$
satisfying condition (iii) of Proposition \ref{thobs} as a decentralized
critical observer for $\mathcal{N}$. Fig. \ref{fig7} shows a possible
implementation architecture for $\mathrm{Obs}^{d}(\mathcal{N})$. Observer $\mathrm{Obs}^{d}(
\mathcal{N})$ can be obtained as a bank on $N$ local observers $\mathrm{Obs}
(M_{i})$ that act asynchronously. Each local observer $\mathrm{Obs}(M_{i})$
takes as input the trace $w_{i}\in L(M_{i})$ generated by $M_{i}$ in
response to the input word $w$, and sends the output boolean values $y_{i}$
to the OR (static) block. The OR block acts as a coordinator and whenever it
receives one or more boolean values $y_{i}$ as inputs, it outputs boolean
value $y$ as the logical operation \textit{or} among $y_{i}$. Note that this
architecture does not require the explicit construction of the parallel
composition of local observers $\mathrm{Obs}(M_{i})$. \newline
Since space and time computational complexity in constructing $\mathrm{Obs}
^{d}(\mathcal{N})$ is $O(2^{n_{\max }N})$, a direct consequence of Theorem 
\ref{thObsDistr} is:

\begin{corollary}
\label{corogio} An upper bound to space and time computational complexity in constructing $
\mathrm{Obs}(\mathbf{M}(\mathcal{N}))$ is $O(2^{n_{\max}N})$.
\end{corollary}

\begin{figure}[tbp]
\begin{center}
\includegraphics[scale=0.33]{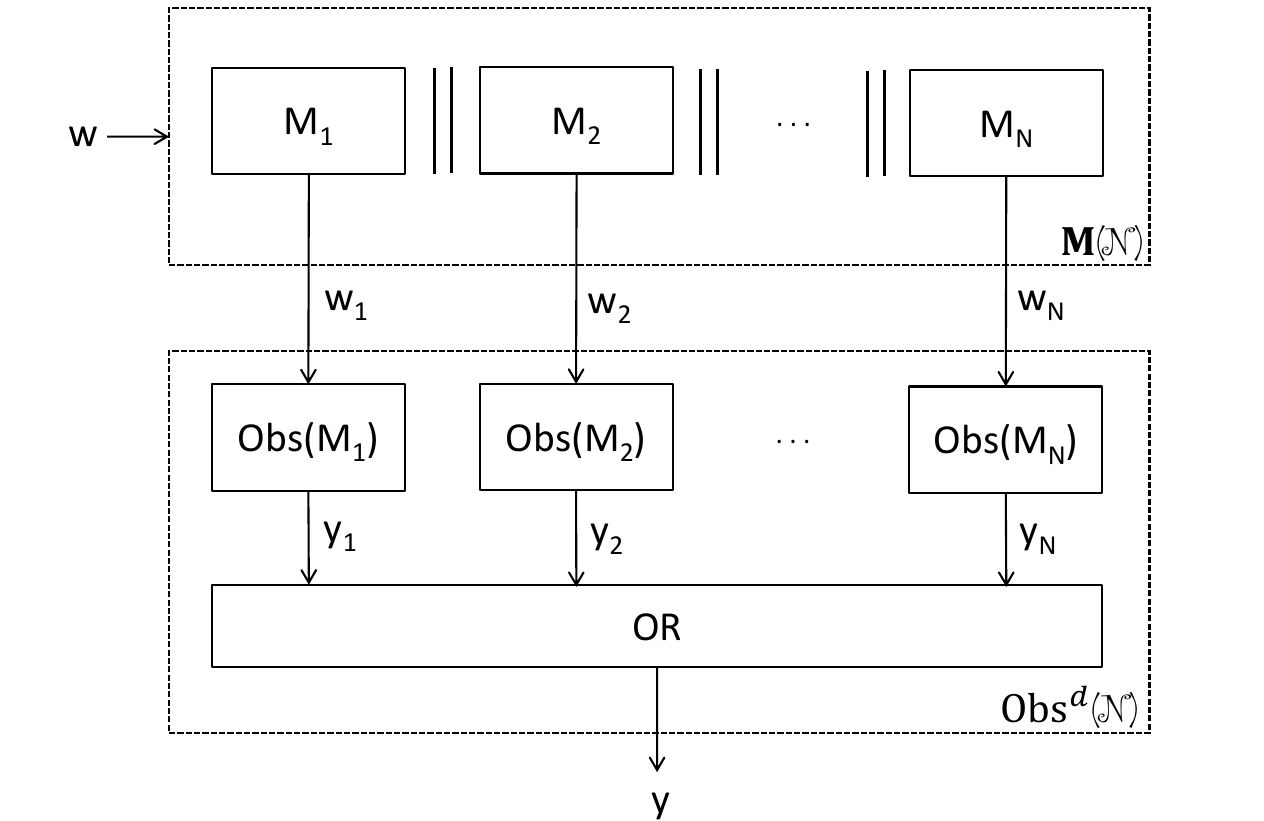}
\end{center}
\caption{A possible architecture for the decentralized observer $\mathrm{Obs}
^{d}(\mathcal{N})$.}
\label{fig7}
\end{figure}

Note that the above upper bound is lower 
than $O(2^{2^{N\log (n_{\max })}})$. \\
From Proposition \ref{thobs} and Theorem \ref{thObsDistr}, one way to check critical observability and to design decentralized observers of $
\mathbf{M}(\mathcal{N})$ is illustrated in Algorithm \ref{alg1}.

\begin{algorithm}
\caption{Check of critical observability of $\mathbf{M}(\mathcal{N})$.}
\label{alg1}
\scriptsize
\begin{algorithmic}[1]
\STATE Construct $N$ local observers $\Obs(M_i)$;
\STATE Compose the $N$ local observers $\Obs(M_i)$ to get $\Obs^{d}(\mathcal{N})$;
\STATE Apply Proposition \ref{thobs} to $\Obs^{d}(\mathcal{N})$.
\end{algorithmic}
\end{algorithm}

Algorithm \ref{alg1} has space and time computational complexity $O(2^{n_{\max }N})$. 
It can be improved from the computational point of view, because:\newline
(D1) It constructs the whole local observer $\mathrm{Obs}(M_{i})$ for each $
M_{i}$. A more efficient algorithm would construct, for each $M_{i}$, only
the sub--FSM of $\mathrm{Obs}(M_{i})$ that is interconnected with the other
local observers in $\mathrm{Obs}^{d}(\mathcal{N})$. \newline
(D2) It constructs the whole observer $\mathrm{Obs
}^{d}(\mathcal{N})$, which is not required at the implementation layer (see
Fig. \ref{fig7}). A more efficient algorithm would check critical
observability on the basis of local observers. \newline
(D3) It first constructs $\mathrm{Obs}^{d}(\mathcal{N})$ before checking if
states $z$ of $\mathrm{Obs}^{d}(\mathcal{N})$ satisfy condition (iii) of
Proposition \ref{thobs}. A more efficient algorithm would conclude that $
\mathcal{N}$ is not critically observable when the first state $z$ not
satisfying condition (iii) of Proposition \ref{thobs}, shows up.\newline
In order to cope with the aforementioned drawbacks, we now present a
procedure that \textit{integrates each step of Algorithm \ref{alg1} in one
algorithm}. This procedure extends on--the--fly algorithms for 
verification and control of FSMs (see e.g. \cite{onthefly3,onthefly2}) 
to the design of decentralized critical observers and
is reported in Algorithm \ref{alg}. \\
\begin{algorithm}
\caption{Integrated design of decentralized observers}
\label{alg}
\scriptsize
\begin{algorithmic}[1]
\STATE \textbf{Input:} FSMs $M_i=(X_i,X^0_i,\Sigma_i,\delta_i,C_i)$, with $i\in [1;N]$;
\STATE \textbf{Init:} $X_{\Obs,i}:=\{X_i^0\}$, $X^{0}_{\Obs,i}:=\{X_i^0\}$, $Y_{\Obs,i}:=\{0,1\}$ for any $i\in [1;N]$; $X_{\Obs}:=\{(X_1^0,X_2^0,...,X_N^0)\}$; $X_{\Obs}^{\temp}:=X_{\Obs}$;
\WHILE{$X_{\Obs}^{\temp}\neq\varnothing$}
\STATE $Z_{\Obs}^{\temp}:=\varnothing$
\FORALL{$(z_1,z_2,...,z_N) \in X_{\Obs}^{\temp}$}
\FORALL{$\sigma \in \Sigma_1 \cup \Sigma_2 \cup ... \cup \Sigma_N$}
\FORALL{$i\in[1;N]$}
\IF{$\delta_i(z_i,\sigma)$\textbf{ is defined}}
\STATE $z_i^+ :=\delta_i(z_i,\sigma)$;
\ELSE
\STATE $z_i^+ :=z_i$;
\ENDIF
\ENDFOR
\IF{$\delta_{\Obs^d}((z_1,z_2,...,z_N),\sigma)=\{(z_1^+,z_2^+,...,z_N^+)\}$}
\IF{$(z_1^+,z_2^+,...,z_N^+)\notin X_{\Obs}$}
\IF{$[z_1^+ \times z_2^+ \times ... \times z_N^+\nsubseteq C_{1,2,...,N}] \wedge [z_1^+ \times z_2^+ \times ... \times z_N^+\nsubseteq X_{1,2,...,N}\backslash C_{1,2,...,N}]$}
\STATE BREAK: $\mathbf{M}(\mathcal{N})$ is not critically observable;
\ENDIF
\STATE $Z_{\Obs}^{\temp}:=Z_{\Obs}^{\temp}\cup \{(z_1^+,z_2^+,...,z_N^+)\}$;
\FORALL{$i\in [1;N]$}
\IF{$[z_i^+\neq \varnothing ] \wedge [z_i^+\notin X_{\Obs,i}]$}
\STATE $X_{\Obs,i}:=X_{\Obs,i} \cup\{z_i^+\}$;
\STATE $\delta_{\Obs,i}(z_i,\sigma):=\{z_i^+\}$;
\IF{$z_i^+\subseteq C_i$}
\STATE $H_{\Obs,i}(z_i^+):=1$;
\ELSE
\STATE $H_{\Obs,i}(z_i^+):=0$;
\ENDIF
\ENDIF 
\ENDFOR
\ENDIF
\ENDIF
\ENDFOR
\ENDFOR
\STATE $X_{\Obs}:=X_{\Obs}\cup Z_{\Obs}^{\temp}$, $X_{\Obs}^{\temp}:=Z_{\Obs}^{\temp}$;
\ENDWHILE
\STATE \textbf{output:} $\mathbf{M}(\mathcal{N})$ is critically observable;\\Projected local observers $\pi|_{\Obs(M_i)}(\Obs^d(\mathcal{N}))=(X_{\Obs,i},X^{0}_{\Obs,i},\Sigma_i,\delta_{\Obs,i},Y_{\Obs,i},H_{\Obs,i})$.
\end{algorithmic}
\end{algorithm}
Algorithm \ref{alg} makes use of the notion of projected
local observers. The projection $\pi |_{\mathrm{Obs}(M_{i})}(\mathrm{Obs}
^{d}(\mathbf{M}(\mathcal{N})))$ of $\mathrm{Obs}^{d}(\mathbf{M}(\mathcal{N}))
$ onto $\mathrm{Obs}(M_{i})$, as in (\ref{ObsMi}), is defined as the FSM $
\mathrm{Ac}(X_{\mathrm{Obs},i}^{\prime },X_{\mathrm{Obs},i}^{0},\Sigma
_{i},\delta _{\mathrm{Obs},i},Y_{\mathrm{Obs},i},H_{\mathrm{Obs},i}$) where $
X_{\mathrm{Obs},i}^{\prime }$ contains states $z_{i}\in X_{\mathrm{Obs},i}$
for which there exist states $z_{j}\in X_{\mathrm{Obs},j}$, $j\in \lbrack
1;N],j\neq i$ such that $(z_{1},z_{2},...,z_{N})$ is a state of $\mathrm{Obs}
^{d}(\mathbf{M}(\mathcal{N}))$. \newline
The input of Algorithm \ref{alg} is the collection of FSMs $M_{i}$ of $
\mathcal{N}$. The output is the collection of projected local observers $\pi
|_{\mathrm{Obs}(M_{i})}(\mathrm{Obs}^{d}(\mathcal{N}))$ if $\mathbf{M}(
\mathcal{N})$ is critically observable. In line 2, the initial state and the output set
of the projected local observers are defined and their sets of
states $X_{\mathrm{Obs},i}$ are initialized to contain only the initial
states. At each iteration, the
algorithm processes candidate new states of $\pi |_{\mathrm{Obs}(M_{i})}(
\mathrm{Obs}^{d}(\mathcal{N}))$ and adds them to $X_{\mathrm{Obs},i}$
whenever they are compatible with the parallel composition of $\pi |_{
\mathrm{Obs}(M_{j})}(\mathrm{Obs}^{d}(\mathcal{N}))$ with $j\neq i$. For
each aggregate $(z_{1},z_{2},...,z_{N})$ in the temporary set $X_{\mathrm{Obs
}}^{\mathrm{temp}}$ and for each $\sigma \in \Sigma _{1}\cup \Sigma _{2}\cup
...\cup \Sigma _{N}$ (lines 5 and 6), first, successors $z_{i}^{+}$ of
states $z_{i}$ in $M_{i}$ are computed (lines 7--12); in particular, if 
$\delta_i (x_i,\sigma)$ is defined (note that maps $\delta_i$ are in general partial) then $z_{i}^{+}$ is set in line 9 
to $\delta_i (x_i,\sigma)$; otherwise it is set in line 11 to $z_i$. If in line 14,
according to the definition of the transition map of $\mathrm{Obs}^{d}(
\mathcal{N})$, there is a transition in $\mathrm{Obs}^{d}(\mathcal{N})$ from
aggregate $(z_{1},z_{2},...,z_{N})$ to aggregate $%
(z_{1}^{+},z_{2}^{+},...,z_{N}^{+})$ with input label $\sigma $, then $
(z_{1}^{+},z_{2}^{+},...,z_{N}^{+})$ is further processed. Note that we are
not storing information computed in line 14 on the transition map of $
\mathrm{Obs}^{d}(\mathcal{N})$ but only states of $\mathrm{Obs}^{d}(\mathcal{
N})$ which cannot be avoided for computing the projections $\pi |_{\mathrm{
Obs}(M_{i})}(\mathrm{Obs}^{d}(\mathcal{N}))$. By this fact, Algorithm \ref
{alg} overcomes drawback (D2) of Algorithm \ref{alg1}. Algorithm \ref{alg}
first checks in line 15 if the aggregate $(z_{1}^{+},z_{2}^{+},...,z_{N}^{+})
$ was not processed before. If so, line 16 is processed. 
By definition of output function $H_{\mathrm{Obs}^{d}}$, condition $
(z_{1}^{+},z_{2}^{+},...,z_{N}^{+})\nsubseteq X_{1,2,...,N}\backslash
C_{1,2,...,N}$ implies $H_{\mathrm{Obs}
^{d}}(z_{1}^{+},z_{2}^{+},...,z_{N}^{+})=1$ which, combined with condition $
(z_{1}^{+},z_{2}^{+},...,z_{N}^{+})\nsubseteq C_{1,2,...,N}$, implies that
condition (iii) of Proposition \ref{thobs} is not satisfied. Hence, by
applying Proposition \ref{thobs}, if $(z_{1}^{+},z_{2}^{+},...,z_{N}^{+})$
satisfies such a condition, the algorithm immediately terminates in line 17,
concluding that $\mathcal{N}$ is not critically observable. By this fact,
Algorithm \ref{alg} overcomes drawback (D3) of Algorithm \ref{alg1}. If $
(z_{1}^{+},z_{2}^{+},...,z_{N}^{+})$ does not satisfy condition in line 16,
it is added to $Z_{\mathrm{Obs}}^{\mathrm{temp}}$ in line 19; the set of
states $X_{\mathrm{Obs},i}$ and the transition map $\delta _{\mathrm{Obs},i}$
of $\pi |_{\mathrm{Obs}(M_{i})}(\mathrm{Obs}^{d}(\mathcal{N}))$ are updated
in lines 22 and 23. Note that since $\delta _{\mathrm{Obs},i}$ is updated
only if condition in line 14 holds, then Algorithm \ref{alg} constructs
step--by--step $\pi |_{\mathrm{Obs}(M_{i})}(\mathrm{Obs}^{d}(\mathcal{N}))$
and not $\mathrm{Obs}(M_{i})$. By this fact, Algorithm \ref{alg} overcomes
drawback (D1) of Algorithm \ref{alg1}. The outputs of states $z_{i}^{+}$ are
set in lines 24--27. 
In line 35, set $X_{\mathrm{Obs}}$ is updated to $X_{\mathrm{Obs}}\cup Z_{
\mathrm{Obs}}^{\mathrm{temp}}$ and set $X_{\mathrm{Obs}}^{\mathrm{temp}}$ to 
$Z_{\mathrm{Obs}}^{\mathrm{temp}}$; the next iteration then starts.
Algorithm \ref{alg} terminates when there are no more states in $X_{\mathrm{
Obs}}^{\mathrm{temp}}$ to be processed (see line 5) or condition in line 16
is satisfied. From the above explanation, it is clear that Algorithm \ref
{alg} terminates in a finite number of states. Moreover it is also clear
that in the worst case, the computational complexity of Algorithm \ref{alg}
is the same as the one in Algorithm \ref{alg1}, i.e. $O(2^{n_{\max }N})$. This is typical for on--the--fly based algorithms. 
However, there are practical cases in which Algorithm \ref{alg} performs
better than Algorithm \ref{alg1}; a pair of examples are included in the end of the next section.

\section{Model reduction via bisimulation}
\label{sec4}

In this section we propose the use of bisimulation equivalence \cite%
{Milner,Park} to reduce the computational complexity in checking critical
observability and designing observers. We start by recalling the notion of bisimulation
equivalence.

\begin{definition}
\label{bisim} Two FSMs $M_{1}=\left(X_{1},X_{1}^0,\Sigma_1,\delta_1,C_1
\right)$ and $M_2=\left(X_2,X_{2}^0,\Sigma_2,\delta_2,C_2\right)$ are
bisimilar, denoted by $M_1\cong M_2$, if $\Sigma_1 = \Sigma_2$ and there exists a relation $R
\subseteq X_1 \times X_2$, called bisimulation relation, such that for any $
(x_1,x_2) \in R$ the following conditions are satisfied:\newline
(i) $x_1 \in X_{1}^0$ if and only if $x_2 \in X_{2}^0$;\newline
(ii) for any $\sigma \in \Sigma_1$ such that $\delta_1(x_1,\sigma)\neq
\varnothing$ and for any $x_1^+ \in \delta_1(x_1,\sigma)$ there exists $
x_2^+ \in \delta_2(x_2,\sigma)$ such that $(x_1^+,x_2^+) \in R$;\newline
(iii) for any $\sigma \in \Sigma_2$ such that $\delta_2(x_2,\sigma)\neq
\varnothing$ and for any $x_2^+ \in \delta_2(x_2,\sigma)$ there exists $
x_1^+ \in \delta_1(x_1,\sigma)$ such that $(x_1^+,x_2^+) \in R$;\newline
(iv) $x_1 \in C_1$ if and only if $x_2 \in C_2$. 
\end{definition}

The above notion of bisimulation equivalence differs from the classical one 
\cite{Milner,Park} because of the additional condition (iv). This condition is needed because for two states $x_1$ and $x_2$ to be considered as equivalent, they have to be either both critical or both noncritical states. 
Using this notion of bisimulation equivalence, we get the following result:

\begin{proposition}
\label{propnew} If FSMs $M_{1}$ and $M_{2}$ are bisimilar then: \newline
(i) $M_{1}$ is critically observable if and only if $M_{2}$ is critically
observable;\newline
(ii) An FSM $\mathrm{Obs}$ is a critical observer for $M_{1}$ if and only if it is
a critical observer for $M_{2}$.
\end{proposition}

\begin{proof}
Set $M_i=(X_i,X_i^0,\Sigma_i,$ $\delta_i,C_i)$, $i=1,2$ and let $R$ be a bisimulation relation between $M_1$ and $M_2$. \\
Proof of (i). By contradiction assume
that $M_1$ is critically observable and $M_2$ is not critically observable.
Hence, there exist a pair of state runs of $M_2$ with initial states $x_2^0,
\overline{x}_2^0\in X_{2}^{0}$, common trace $w \in L(M_2)$, and states $
x_2\in \hat{\delta}_{2}(x_2^0,w)$, $\overline{x}_2 \in \hat{\delta}_{2}(
\overline{x}_2^0,w)$ such that 
\begin{equation}  \label{qaa}
x_2 \in C_2 \wedge \overline{x}_2 \in X_{2} \backslash C_2.
\end{equation}
Since $M_1 \cong M_2$, there exist a pair of state runs of $M_1$ with initial states $x_1^0,
\overline{x}_1^0\in X_{1}^{0}$, common trace $w \in L(M_1) $ and states $x_1
\in \hat{\delta}_{1}(x_1^0,w)$, $\overline{x}_1 \in \hat{\delta}_{1}(
\overline{x}_1^0,w)$ such that $(x_1,x_2),(\overline{x}_1,\overline{x}_2)\in
R$ which, by (\ref{qaa}) and condition (iv) in Definition \ref{bisim},
implies $x_1\in C_1$ and $\overline{x}_1\in X_{1} \backslash C_1$. Thus, $M_1
$ is not critically observable and a contradiction holds. \\
Proof of (ii). 
By contradiction assume that
$\mathrm{Obs}$ is a critical observer for $M_{1}$ but not for $M_2$. By Definition \ref{defobsabstract}, there exist a state run 
$r_2: x^0_2 \rTo^{\sigma^1} x^1_2$ $\rTo^{\sigma^2}
x^2_2 \rTo^{\sigma^3} x^3_2 \, ... $ of $M_2$ and the corresponding (unique) state
run $r_\mathrm{Obs}: z^0 \rTo^{\sigma_1} z^1 \rTo^{\sigma_2} z^2 \rTo
^{\sigma_3} z^3 \, ... $ of $\mathrm{Obs}$ such that
\begin{equation}  \label{qaaq}
[H_\mathrm{Obs}
(z^{\overline{i}})=1\wedge x^{\overline{i}}_2\notin C_2] \vee [H_\mathrm{Obs}(z^{\overline{i}})=0 \wedge x^{\overline{i}}_2\in C_2 ],
\end{equation}
for some $\overline{i}$. 
Suppose first $H_\mathrm{Obs} (z^{\overline{i}})=1\wedge x^{\overline{i}}_2\notin C_2$. By Definition \ref{bisim}, there exists a state run $r_1: x^0_1 \rTo^{\sigma^1} x^1_1$ $\rTo^{\sigma^2}
x^2_1 \rTo^{\sigma^3} x^3_1 \, ... $ of $M_1$ such that $(x_1^i,x_2^i)\in R$ for all $i$. In particular for $i=\overline{i}$ we get by condition (iv) of Definition \ref{bisim} that $x^{\overline{i}}_1\notin C_1$ from which, $\mathrm{Obs}$ is not a critical observer for $M_1$ and a contradiction holds. The second case in (\ref{qaaq}) can be proven by using the same arguments.
\end{proof}

Space and time computational complexities in checking bisimulation equivalence between $M_{1}$ and $M_{2}$
with $|X_{1}|=n_{1}$ and $|X_{2}|=n_{2}$ states are $O(n_{1}^{2}+n_{2}^{2})$
and $O((n_{1}^{2}+n_{2}^{2})\log (n_{1}+n_{2}))$, respectively, see e.g. 
\cite{BisAlg,Piazza}. Bisimulation equivalence is an equivalence relation on
the class of FSMs. We now define the network of FSMs $\mathcal{N}^{\min }$
as the quotient of the original network $\mathcal{N}$ induced by the
bisimulation equivalence. More precisely given $\mathcal{N}$, define the
following equivalence classes induced by the bisimulation equivalence:

$\mathcal{E}_1=\{M_{\mathbf{i}(1,1)}, M_{\mathbf{i}(1,2)},..., M_{\mathbf{i}
(1,n^1)}\},$\newline
$\mathcal{E}_2=\{M_{\mathbf{i}(2,1)}, M_{\mathbf{i}(2,2)},..., M_{\mathbf{i}
(2,n^2)}\},$\newline
$...,$ \newline
$\mathcal{E}_{N^{\min}}=\{M_{\mathbf{i}(N^{\min},1)}, M_{\mathbf{i}
(N^{\min},2)},..., M_{\mathbf{i}(N^{\min},n^{N^{\min}})}\}$, 

such that $\mathcal{N}=\{M_{\mathbf{i}(k,j)}\}_{j\in [1;n^k],k\in [1;N^{\min}]}$ and $M_{\mathbf{i}(k,j_{1})},$ $M_{\mathbf{i}(k,j_{2})}\in 
\mathcal{E}_{k}$ if and only if $M_{\mathbf{i}(k,j_{1})}\cong M_{\mathbf{i}
(k,j_{2})}$. 
Denote by $M_{\mathbf{i}_{k}^{\min }}\in \mathcal{E}_{k}$ a representative
of the equivalence class $\mathcal{E}_{k}$ and define the network of FSMs 
\begin{equation*}
\mathcal{N}^{\min }=\{M_{\mathbf{i}_{1}^{\min }},M_{\mathbf{i}_{2}^{\min
}},...,M_{\mathbf{i}_{N^{\min }}^{\min }}\}
\end{equation*}
with $\mathbf{M}(\mathcal{N}^{\min })=M_{\mathbf{i}_{1}^{\min }}||M_{\mathbf{
i}_{2}^{\min }}||...||M_{\mathbf{i}_{N^{\min }}^{\min }}$. 
The following result holds.
\begin{theorem}
\label{bisimth}
If $C_i = \varnothing$ for all $i\in [1;N]$ or $C_i = X_i$ for all $i\in [1;N]$ then $\mathbf{M}(\mathcal{N}) \cong \mathbf{M}(\mathcal{N}^{\min })$.
\end{theorem}

\begin{proof}
From Proposition 3.1.5 in \cite{ThesisPezzuti},
\begin{equation}
\label{eq1}
[ [M_1 \cong M_2 ] \wedge [ M_3 \cong M_4 ]] \Rightarrow [M_1||M_3 \cong M_2||M_4 ].
\end{equation}
If we show that
\begin{equation}
\label{eq2}
[M_1 \cong M_2 ] \Rightarrow [ M_1||M_2 \cong M_1 ],
\end{equation}
then the result holds as a straightforward application of (\ref{eq1}) and (\ref{eq2}) and by the definition of $\mathbf{M}(\mathcal{N})$ and $\mathbf{M}(\mathcal{N}^{\min })$. Let $M_i =(X_i, X_i^0, \Sigma_i, \delta_i, C_i)$ for $i=1,2$, and $M_1||M_2 =(X,X^0,\Sigma, \delta, C)$. Define $R\subseteq X \times X_2$ such that $((x_1,x_2),x'_1)\in R$ if and only if $x_1=x_1'$. Consider any $((x_1,x_2),x'_1)\in R$. Condition (i) of Definition 
\ref{bisim} holds by definition of $R$ and of initial states in Definition \ref{composition}. As far as for condition (ii), first note that since $M_1 \cong M_2$ then $\Sigma=\Sigma_1=\Sigma_2$ and the parallel composition $M_1 \Vert M_2$ boils down to the product composition $M_1 \times M_2$ (see e.g. \cite{Cassandras}), i.e. only synchronized transitions are allowed; consider now any $\sigma \in \Sigma$ such that $\delta((x_1,x_2),\sigma)$ is defined and any $(x^+_1,x_2^+) \in \delta((x_1,x_2),\sigma)$. 
By Definition \ref{composition} and since $\Sigma_1 \cong \Sigma_2$ then $x_1^+\in \delta_1(x_1,\sigma)$ from which, by setting $x^{\prime,+}_1=x^+_1$ we get $((x^+_1,x_2^+),x^{\prime,+}_1)\in R$. We now show condition (iii) of Definition \ref{bisim}. Consider any $\sigma \in \Sigma_1$ such that $\delta_1(x_1,\sigma)$ is defined and any $x^{\prime,+}_1 \in \delta_1(x_1,\sigma)$. Select $x_1^+= x_1^{\prime,+}$. Since $\Sigma_1 \cong \Sigma_2$ then there exists $x_2^+\in \delta_{2}(x_2,\sigma)$ such that the pair $(x_1^+,x_2^+)$ is in a bisimulation relation between $\Sigma_1$ and $\Sigma_2$; this in turn implies, by Definition \ref{composition}, that $(x_1^+,x_2^+)$ is a state in $\delta((x_1,x_2),\sigma)$; since by construction $((x^+_1,x_2^+),x^{\prime,+}_1)\in R$ then condition (iii) holds. Condition (iv) is trivially satisfied by the assumptions placed on the sets of critical states $C_i$.
\end{proof}

The above result extends bisimulation theory based model reduction, traditionally given in the community of formal methods for \textit{single FSMs}, see e.g. \cite{ModelChecking} and the references therein, to \textit{networks of FSMs}.  
It establishes connections between the networks of FSMs $\mathbf{M}(\mathcal{N})$ and $\mathbf{M}(\mathcal{N}^{\min })$ under the assumption that each set $C_i$ is either empty or coincides with the set of states $X_i$. This assumption corresponds in fact to drop condition (iv) from Definition \ref{bisim}, thus obtaining the standard definition of bisimulation as in e.g. \cite{Milner,Park}. If this assumption is removed, the following example shows 
that the statement of Theorem \ref{bisimth} is not true in general. 

\begin{example}
\label{exampleContra}
Consider the FSMs $M_1$ depicted in Fig. \ref{fig1contra} and $M_2 = M_1$. Of course, $M_1 \cong M_2$. FSM $M_1 || M_2$ is depicted in Fig. \ref{fig2contra}. While in $M_1 || M_2$ there are two transitions from critical states $(2,3)$ and $(3,2)$ to critical states $(4,5)$ and $(5,4)$, respectively, there is no transition in $M_1$ from a critical state to a critical state. Hence, this is enough to conclude that FSMs $M_1 || M_2$ and $M_1$ are not bisimilar.
\end{example}

\begin{figure}[tbp]
\begin{center}
\includegraphics[scale=0.25]{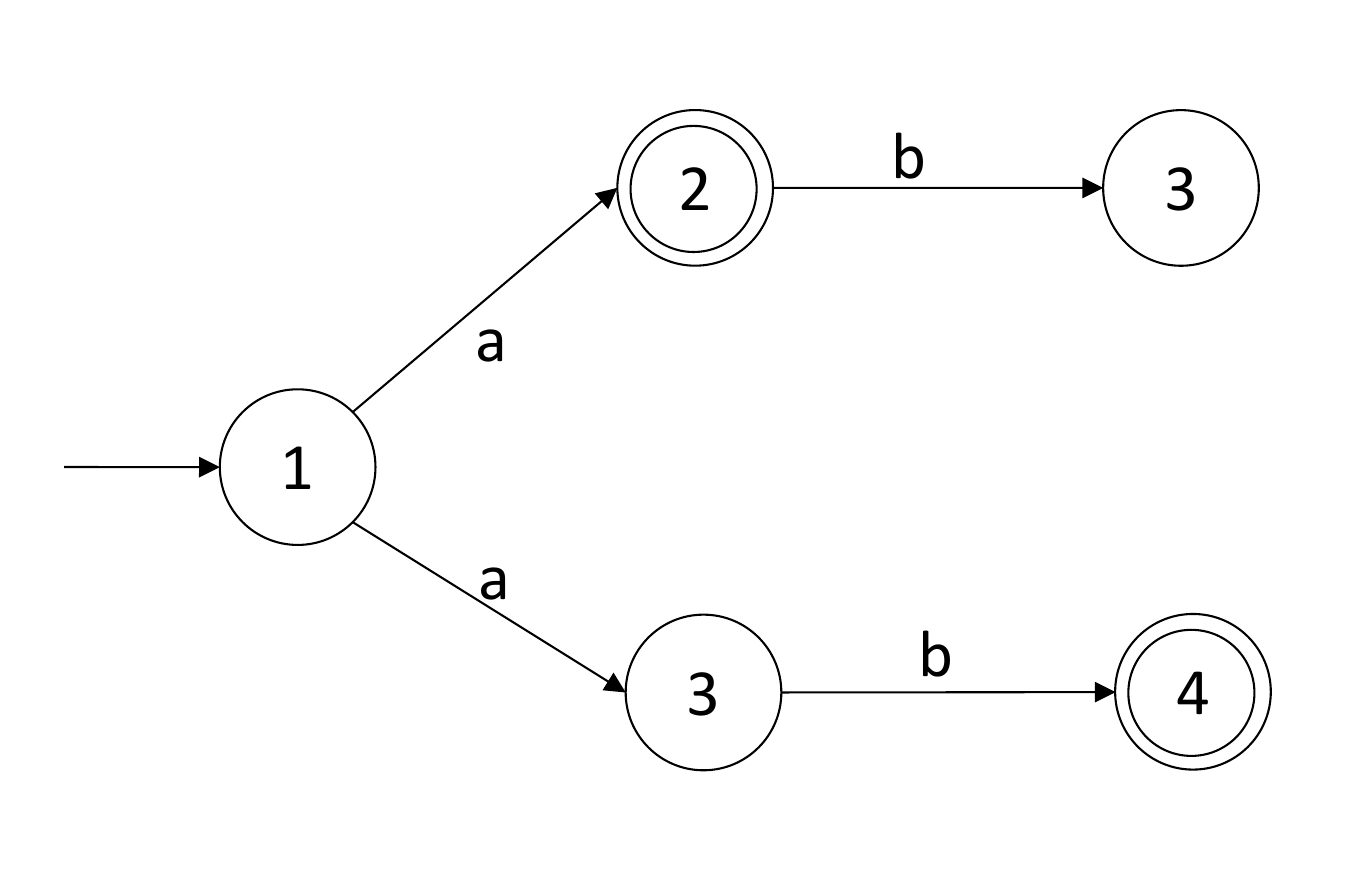}
\end{center}
\caption{FSM $M_1$ in Example \ref{exampleContra}.}
\label{fig1contra}
\end{figure}

\begin{figure}[tbp]
\begin{center}
\includegraphics[scale=0.25]{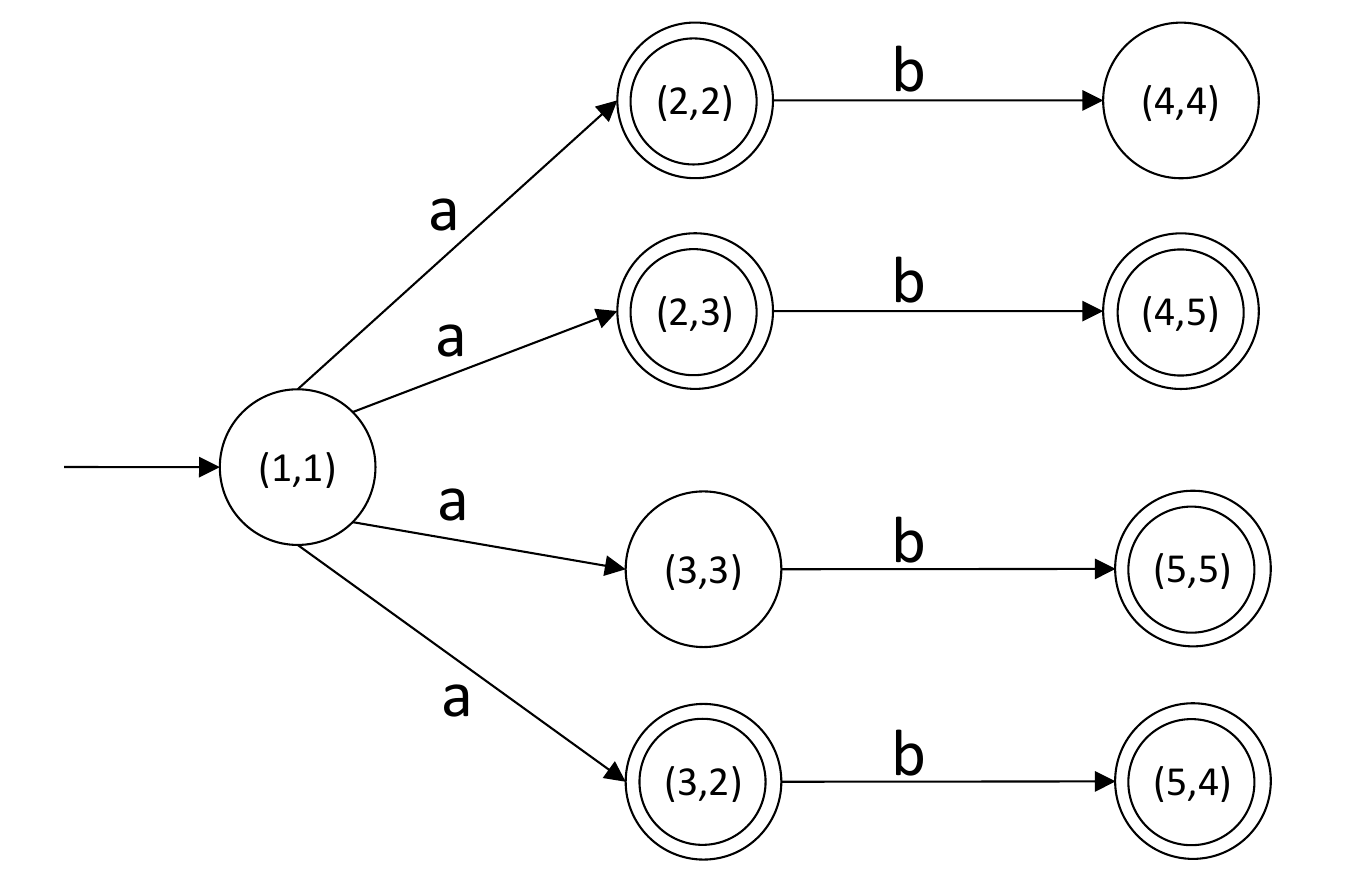}
\end{center}
\caption{FSM $M_1 \Vert M_2$ in Example \ref{exampleContra}.}
\label{fig2contra}
\end{figure}

However, we now show that the proposed reduced network $\mathbf{M}(\mathcal{N}^{\min })$ preserves the critical observability properties of the original network $\mathbf{M}(\mathcal{N})$. 
The forthcoming results rely upon the following technical results.

\begin{lemma}
\label{lemmaObs} Consider FSMs $M_i=(X_i,X^{0}_i,\Sigma_i,\delta_i,C_i)$
with $i\in [1;3]$ and suppose that $M_2$ and $M_3$ are bisimilar. Then an FSM $\mathrm{Obs}$ is a critical observer for $M_1 || M_2$ if and only if it is a
critical observer for $M_1 || M_2 ||M_3$.
\end{lemma}

\begin{proof}
Let be $\mathrm{Obs}=(X_\mathrm{Obs}
,X^{0}_\mathrm{Obs},\Sigma_\mathrm{Obs},$ $\delta_\mathrm{Obs},Y_\mathrm{Obs}
,H_\mathrm{Obs})$ and $R_{23}$ be a bisimulation relation between $M_2$ and $
M_3$. (Sufficiency.) By contradiction assume that $\mathrm{Obs}$ is not a
critical observer for $M_1 || M_2 ||M_3$. By Definition \ref{defobsabstract}
, there exist a state run $r:(x^0_1,x^0_2,x^0_3) \rTo^{\sigma^1}
(x^1_1,x^1_2,x^1_3) \rTo^{\sigma^2} \,$ $... \, \rTo^{\sigma^n}
(x^n_1,x^n_2,x^n_3)
$ of $M_1 || M_2 ||M_3$ and the corresponding state run $r_{\mathrm{Obs}
}:\quad z^0 \rTo^{\sigma^1} z^1 \rTo^{\sigma^2} \, ... \, \rTo^{\sigma^n} z^n
$ of $\mathrm{Obs}$ such that (case 1) $[(x^n_1,x^n_2,x^n_3)\notin
C_{1,2,3}\wedge H_\mathrm{Obs}(z^n)=1]$  or (case 2) $[(x^n_1,x^n_2,x^n_3)
\in C_{1,2,3}\wedge H_\mathrm{Obs}(z^n)=0]$. Construct the sequence 
$r_{1,2}: (x^0_1,x^0_2) \rTo^{\sigma^1}$ $(x^1_1,x^1_2) \rTo^{\sigma^2} \,
... \, \rTo^{\sigma^n} (x^n_1,$ $x^n_2)$, 
where $r_{1,2}$ has been obtained by removing the third component in the
states of run $r$. It is readily seen that $r_{1,2}$ is a state run of $M_1
|| M_2$. Construct the sequence $\hat{r}_{1,2}:\quad (x^0_1,\hat{x}^0_2) \rTo
^{\sigma^1} (x^1_1,\hat{x}^1_2) \rTo^{\sigma^2} \, ... \, \rTo^{\sigma^n}
(x^n_1,\hat{x}^n_2) $ 
such that $(\hat{x}^i_2,x^i_3)\in R_{23}$ for any $i\in [0;n]$. By
construction, $\hat{r}_{1,2}$ is a state run of $M_1||M_2$. We start by
considering case 1. Since $(x^n_1,x^n_2,x^n_3)\notin C_{1,2,3}$ then $
x^n_1\notin C_1$ and $x^n_2\notin C_2$ from which, the last state $
(x^n_1,x^n_2)$ of run $r_{1,2}$ is such that $(x^n_1,x^n_2)\notin C_{1,2}$.
Since $H_\mathrm{Obs}(z^n)=1$, FSM $\mathrm{Obs}$ is not a critical observer
for $M_1 || M_2$ and a contradiction holds. We now consider case 2. Since $
(x^n_1,x^n_2,x^n_3)\in C_{1,2,3}$ then (case 2.1) $[x^n_1\in C_1 \vee
x^n_2\in C_2]$ or (case 2.2) $x^n_3\in C_3$. We start by considering case
2.1. Since $[x^n_1\in C_1 \vee x^n_2\in C_2]$ or equivalently, the last
state $(x^n_1,x^n_2)$ of run $r_{1,2}$ is such that $(x^n_1,x^n_2)\in C_{1,2}
$ and by assumption $H_\mathrm{Obs}(z^n)=0$, a contradiction holds. We
conclude with case 2.2. Since $x^n_3\in C_3$, by definition of run $\hat{r}
_{1,2}$, state $\hat{x}^n_2\in C_2$ which implies that the last state $
(x^n_1,\hat{x}^n_2)$ of run $\hat{r}_{1,2}$ is such that $(x^n_1,\hat{x}
^n_2)\in C_{1,2}$. This last condition combined with the assumed condition $
H_\mathrm{Obs}(z^n)=0$, leads to a contradiction, as well. (Necessity.) By
contradiction assume that $\mathrm{Obs}$ is not a critical observer for $M_1
|| M_2$. By Definition \ref{defobsabstract}, there exist a state run $
r:\quad (x^0_1,x^0_2) \rTo^{\sigma^1} (x^1_1,x^1_2) \rTo^{\sigma^2} \, ...
\, \rTo^{\sigma^n} (x^n_1,x^n_2) $ of $M_1 || M_2$ and the corresponding
state run $r_{\mathrm{Obs}}:\quad z^0 \rTo^{\sigma^1} z^1 \rTo^{\sigma^2} \,
... \, \rTo^{\sigma^n} z^n $ of $\mathrm{Obs}$ such that (case 1) $
[(x^n_1,x^n_2)\notin C_{1,2}\wedge H_\mathrm{Obs}(z^n)=1]$ or (case 2) $
[(x^n_1,x^n_2)\in C_{1,2}\wedge H_\mathrm{Obs}(z^n)=0]$. Construct the
sequence $r^{\prime}: (x^0_1,x^0_2,x^0_3) \rTo^{\sigma^1} (x^1_1,x^1_2,x^1_3) 
\rTo^{\sigma^2} \, ... \, \rTo^{\sigma^n} (x^n_1,x^n_2,x^n_3) $, 
where $(x^i_2,x^i_3)\in R_{23}$ for any $i\in [0;n]$. By construction $
r^{\prime }$ is a state run of $M_1 || M_2 || M_3$. We start by considering
case 1. By condition (iv) of Definition \ref{bisim}, we get $
[\,(x^n_1,x^n_2)\notin C_{1,2}\,]$ iff $[\,x^n_1\notin C_1 \wedge
x^n_2\notin C_2\,]$ iff $[\,x^n_1\notin C_1 \wedge x^n_2\notin C_2 \wedge
x^n_3\notin C_3\,]$ iff $[\,(x^n_1,x^n_2,x^n_3)\notin C_{1,2,3}\,]$. 
Since $H_\mathrm{Obs}(z^n)=1$, FSM $\mathrm{Obs}$ is not a critical observer
for $M_1 || M_2 || M_3$ and a contradiction holds. We now consider case 2.
By condition (iv) of Definition \ref{bisim}, we get $[\,(x^n_1,x^n_2)\in
C_{1,2}\,]$ iff $[\,x^n_1\in C_{1} \vee x^n_2\in C_{2}\,]$ iff $[\,x^n_1\in
C_{1} \vee x^n_2\in C_{2} \vee x^n_3\in C_{3}\,]$ iff $
[\,(x^n_1,x^n_2,x^n_3)\in C_{1,2,3}\,]$. 
Since $H_\mathrm{Obs}(z^n)=0$, FSM $\mathrm{Obs}$ is not a critical observer
for $M_1 || M_2 || M_3$ and a contradiction holds. 
\end{proof}

\begin{corollary}
\label{lemma4_3} Consider FSMs $M_i=(X_i,X^{0}_i,\Sigma_i,\delta_i,C_i)$
with $i\in [1;3]$ and suppose that $M_2$ and $M_3$ are bisimilar. Then $M_1
|| M_2$ is critically observable if and only if $M_1 || M_2 ||M_3$ is
critically observable.
\end{corollary}

\begin{proof}
The result follows by combining Proposition \ref{thobs} and Lemma \ref{lemmaObs}.
\end{proof}

We now have all the ingredients to present the main results of this section.

\begin{theorem}
\label{maintheorem} $\mathbf{M}(\mathcal{N})$ is critically observable
if and only if $\mathbf{M}(\mathcal{N}^{\min})$ is critically observable.
\end{theorem}

\begin{proof}
By applying Corollary \ref{lemma4_3}, for any $M_{\mathbf{i}(N^{\min},j)}\in 
\mathcal{N}\backslash \mathcal{N}^{\min}$, FSM $\mathbf{M}(\mathcal{N}
^{\min})=(M_{\mathbf{i}^{\min}_1}||M_{\mathbf{i}^{\min}_2}||$ $...||M_{
\mathbf{i}^{\min}_{N^{\min}-1}})$ $||M_{\mathbf{i}^{\min}_{N^{\min}}} $ is
critically observable if and only if FSM $(M_{\mathbf{i}^{\min}_1}||M_{
\mathbf{i}^{\min}_2}||...||M_{\mathbf{i}^{\min}_{N^{\min}-1}})||M_{\mathbf{i}
^{\min}_{N^{\min}}}||M_{\mathbf{i}(N^{\min},j)} $ is critically observable
(recall that $M_{\mathbf{i}(N^{\min},j)}\cong M_{\mathbf{i}
^{\min}_{N^{\min}}}$ for any $j\in [1;n^{N^{\min}}]$). Hence, by applying recursively Corollary \ref{lemma4_3}
to all other FSMs $M_{\mathbf{i}(k,j)}\in \mathcal{N}\backslash \mathcal{N}
^{\min}$ and by making use of Proposition \ref{Com+Ass-prop} to properly
rearrange terms in the composed FSM, the result follows.
\end{proof}

\begin{theorem}
\label{thmain2} $\mathrm{Obs}(\mathbf{M}(\mathcal{N}^{\min }))$ is a
critical observer for $\mathbf{M}(\mathcal{N}^{\min })$ if and only if it is
a critical observer for $\mathbf{M}(\mathcal{N})$.
\end{theorem}

\begin{proof}
By applying Lemma \ref{lemmaObs}, $\mathrm{Obs}(\mathbf{M}(\mathcal{N}^{\min
}))$ is a critical observer for $\mathbf{M}(\mathcal{N}^{\min })=(M_{\mathbf{
i}_{1}^{\min }}||M_{\mathbf{i}_{2}^{\min }}||...||$ $M_{\mathbf{i}_{N^{\min
}-1}^{\min }})||M_{\mathbf{i}_{N^{\min }}^{\min }}$ if and only if it is a
critical observer for $(M_{\mathbf{i}_{1}^{\min }}||M_{\mathbf{i}_{2}^{\min
}}||...||M_{\mathbf{i}_{N^{\min }-1}^{\min }})||M_{\mathbf{i}_{N^{\min
}}^{\min }}||M_{\mathbf{i}(N^{\min },j)}$ for any FSMs $M_{\mathbf{i}
(N^{\min },j)}\in \mathcal{N}\backslash \mathcal{N}^{\min }$ (recall that $
M_{\mathbf{i}(N^{\min },j)}\cong M_{\mathbf{i}_{N^{\min }}^{\min }}$ for any $j\in [1;n^{N^{\min}}]$).
Hence, by applying recursively Lemma \ref{lemmaObs} to all other FSMs $M\in 
\mathcal{N}\backslash \mathcal{N}^{\min }$ and by making use of Proposition 
\ref{Com+Ass-prop} to properly rearrange terms in the composed FSM, the
result follows.
\end{proof}

The above results reduce the computational complexity effort since they
show that it is possible to consider the reduced network $\mathcal{N}^{\min }
$ to check critical observability and to design critical observers for the original network $\mathcal{N}$. 
We stress that the model reduction via bisimulation equivalence that we
propose here is performed on the collection of FSMs $M_{i}$ and not on the
FSM $\mathbf{M}(\mathcal{N})$, as done for example in \cite
{Wonham:03}; this may allow a drastic computational complexity reduction
when several bisimilar FSMs are present in the network. \\
Results above and in Section \ref{sec3bis} can be
combined together as illustrated in Algorithm \ref{alg3}. 

\begin{algorithm}
\caption{Integrated design of decentralized observers with model reduction}
\label{alg3}
\scriptsize
\begin{algorithmic}[1]
\STATE Compute the network $\mathcal{N}^{\min}$;
\STATE Apply Algorithm \ref{alg} to $\mathcal{N}^{\min}$;
\IF{$\mathbf{M}(\mathcal{N}^{\min})$ is not critically observable}
\STATE BREAK: $\mathbf{M}(\mathcal{N})$ is not critically observable;
\ELSE 
\STATE $\mathbf{M}(\mathcal{N})$ is critically observable;
\ENDIF
\STATE Define projected local observers $\pi|_{\Obs(M_{\mathbf{i}(k,j)})}(\Obs^d(\mathcal{N}))=\pi|_{\Obs(M_{\mathbf{i}_{k}^{\min}})}(\Obs^d(\mathcal{N}^{\min}))$ for any $j\in [1;n_k],k \in [1;N^{\min}]$.
\end{algorithmic}
\end{algorithm}

As a consequence of Proposition \ref{propnew} (ii), the composition of local
observers computed in line 8 of Algorithm \ref{alg3} is a decentralized
critical observer for $\mathcal{N}$. \newline
We now provide a computational complexity analysis. We focus
on computational complexity with respect to parameters $n_{\max}$, $N$ and $
N^{\min}$. A traditional approach to check critical observability of the
network $\mathcal{N}=\{M_1,M_2,...,M_N\}$ consists in computing $\mathrm{Obs}
(\mathbf{M}(\mathcal{N}))$, whose space and time computational complexity by
Corollary \ref{corogio} is $O(2^{n_{\max}N})$, as reported in the second
column of Table \ref{TableComplex}. Computational complexity analysis of
Algorithm \ref{alg3} follows. In line 1, one needs to check bisimulation
equivalence between any pair of FSMs $M_i,M_j$ in $\mathcal{N}$ whose space
computational complexity is $O(n_{\max}^2\,N)$ and time computational
complexity is $O(n_{\max}^2N^2\log(n_{\max}))$. Space and time computational
complexities associated with line 2 are $O(2^{n_{\max}N^{\min}})$ and those
with line 8 are zero. Resulting computational complexity bounds are reported
in the third column of Table \ref{TableComplex}. For example, for $N=10$, $
N^{\min}=7$ and $n_{\max}=10$, space and time computational complexities in
constructing $\mathrm{Obs}(\mathbf{M}(\mathcal{N}))$ are $2^{100}$ and the
ones in constructing $\mathrm{Obs}^{d}(\mathcal{N}^{\min})$ are $2^{70}$. \\
We conclude this section with two illustrative examples. 
In both examples, the goal is to check if a network $\mathcal{N}$ is critically observable and if so, to
design a decentralized critical observer for $\mathcal{N}$. For this purpose
we apply Algorithm \ref{alg3}. In the sequel, space complexity of an FSM is
computed as $S_{1}+S_{2}$ where $S_{1}$ is the sum of the data needed to be
stored for each transition and $S_{2}$ is the number of output data
associated with states. Date stored for a transition from a state $
(z_{1},z_{2},...,z_{m})$ to a state $(z_{1}^{+},z_{2}^{+},...,z_{m^{+}}^{+})$
with a given input label are counted as $\sum_{i\in \lbrack
1;m]}|z_{i}|+\sum_{i\in \lbrack 1;m^{+}]}|z_{i}^{+}|+1$. Time complexity is
computed as the number of transitions generated in composed FSMs and
observers, which represent macro iterations in the algorithms.

\begin{example}
\label{example1tot} Consider $\mathcal{N}=\{M_1,M_2,M_3,M_4\}$ where FSMs $
M_1$ and $M_2$ are depicted in Fig. \ref{fig1} and FSMs $M_3$ and $M_4$ in
Fig. \ref{fig2}, and apply Algorithm \ref{alg3}. (Line 1) It is easy to see
that FSMs $M_1$ and $M_3$ are bisimilar with bisimulation relation $
R_{13}=\{(1,9),(2,10),$ $(3,11),(3,12),(4,13)\}$ and that FSMs $M_2$ and $M_4
$ are bisimilar with bisimulation relation $R_{24}=\{(5,14),(6,15),$ $
(7,16),(7,17),(7,18),(8,16),(8,$ $17),(8,18)\}$. Equivalence classes induced
by the bisimulation equivalence on $\mathcal{N}$ are $\mathcal{E}
_1=\{M_1,M_3\}$ and $\mathcal{E}_2=\{M_2,M_4\}$. The resulting network $
\mathcal{N}^{\min}$ can be chosen as $\{M_1,M_2\}$. (Line 2) By applying
Algorithm \ref{alg} we get that $\mathcal{N}^{\min}$ is critically
observable. The resulting projected local observers $\pi|_{\mathrm{Obs}
(M_{1})}(\mathrm{Obs}^d(\mathcal{N}^{\min}))$ and $\pi|_{\mathrm{Obs}
(M_{2})}(\mathrm{Obs}^d(\mathcal{N}^{\min}))$ are depicted in Fig. \ref{fig8}. 
(Line 8) Define $\pi|_{\mathrm{Obs}(M_{i})}(\mathrm{Obs}^d(\mathcal{N}))$ as 
$\pi|_{\mathrm{Obs}(M_{1})}(\mathrm{Obs}^d(\mathcal{N}^{\min}))$ for $i=1,3$
and as $\pi|_{\mathrm{Obs}(M_{2})}(\mathrm{Obs}^d(\mathcal{N}^{\min}))$ for $
i=2,4$. Space and time complexities in constructing projected local
observers are $54$ and $8$. 
A traditional approach would first construct explicitly $\mathbf{M}(\mathcal{
N})$ and then construct $\mathrm{Obs}(\mathbf{M}(\mathcal{N}))$. The number
of states of $\mathbf{M}(\mathcal{N})$ is $21$ and the one of $\mathrm{Obs}(
\mathbf{M}(\mathcal{N}))$ is $6$. Resulting space and time complexities are $
633$ and $39$. 
\end{example}

\begin{example}
Consider $\mathcal{N}=\{M_1,M_2,M_3,M_4,M_5\}$ where FSMs $M_i$, $i\in [1;4]$
are as in Example \ref{example1tot} and $M_5$ is depicted in Fig. \ref{fig10}
with $C_5=\{20\}$, and apply Algorithm \ref{alg3}. (Line 1) It is easy to
see that for any $i\in [1;4]$, FSMs $M_5$ and $M_i$ are not bisimilar from
which, $\mathcal{N}^{\min}$ can be chosen as $\{M_1,M_2,M_5\}$. (Line 2) In
the first iteration of Algorithm \ref{alg}, starting from the initial state $
z=(\{1\},\{5\},\{19\})$ with label $b$, state $z^+=(\{2\},\{6\},\{20,21\})$
is reached. Since state $z^+$ does not satisfy the condition in line 12 of
Algorithm \ref{alg}, Algorithm \ref{alg} terminates from which, Algorithm 
\ref{alg3} terminates as well, giving as output that $\mathcal{N}$ is not
critically observable. Data stored in line 1 of Algorithm \ref{alg} are $72$
while those in lines 2--9 are $19$ from which, space complexity at line 12
is $91$. Since Algorithm \ref{alg} terminates at the first iteration and no
transitions are generated, time complexity is evaluated as $0$. A
traditional approach would first construct explicitly $\mathbf{M}(\mathcal{N}
)$ to then construct $\mathrm{Obs}(\mathbf{M}(\mathcal{N}))$. The number of
states of $\mathbf{M}(\mathcal{N})$ is $24$ and the one of $\mathrm{Obs}(
\mathbf{M}(\mathcal{N}))$ is $6$. Resulting space and time complexities are $
895$ and $39$. 
\end{example}

\begin{figure}[tbp]
\begin{center}
\includegraphics[scale=0.24]{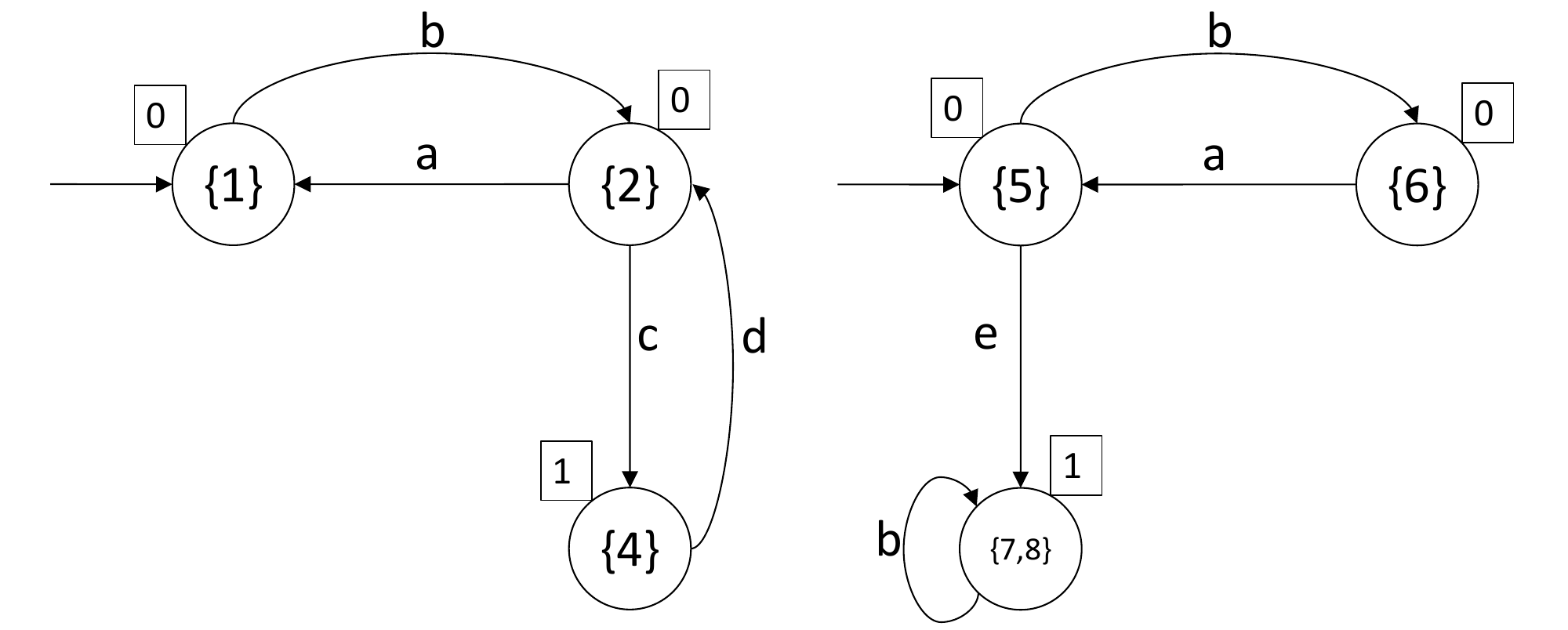}
\end{center}
\caption{Projected local observers $\protect\pi|_{\mathrm{Obs}(M_{1})}(
\mathrm{Obs}^d(\mathcal{N}^{\min}))$ in the left and $\protect\pi|_{\mathrm{
Obs}(M_{2})}(\mathrm{Obs}^d(\mathcal{N}^{\min}))$ in the right.}
\label{fig8}
\end{figure}

\begin{figure}[tbp]
\begin{center}
\includegraphics[scale=0.25]{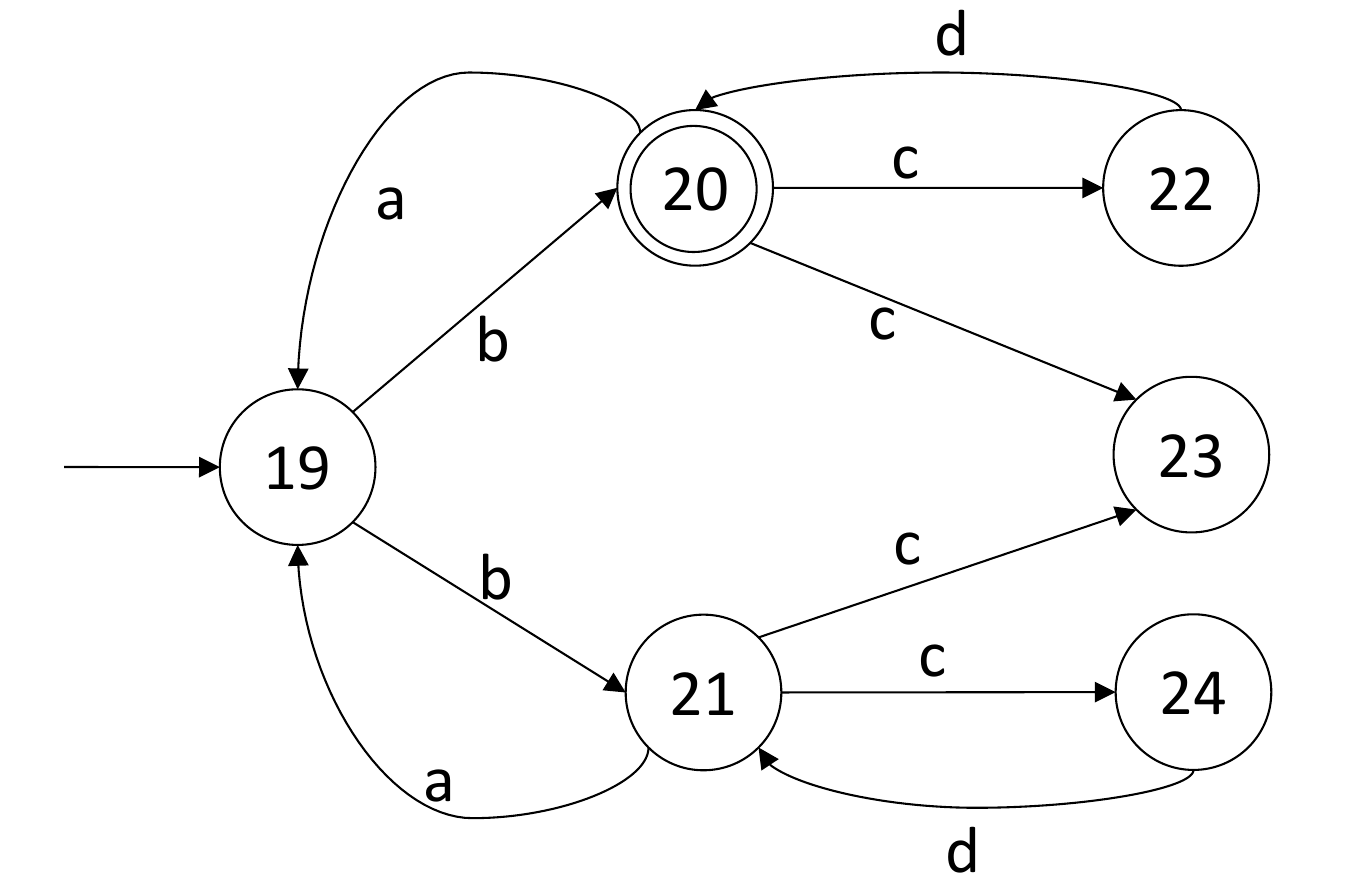}
\end{center}
\caption{FSM $M_5$.}
\label{fig10}
\end{figure}

\begin{table}[h]
\begin{center}
{\tiny 
\begin{tabular}{lll}
\hline
Complexity & $\mathrm{Obs}(\mathbf{M}(\mathcal{N}))$ & $\mathrm{Obs}^{d}(
\mathcal{N}^{\min})$ \\ \hline
Space & $O(2^{n_{\max}N})$ & $O(n_{\max}^2\,N+2^{n_{\max}N^{\min}})$ \\ 
Time & $O(2^{n_{\max}N})$ & $O(n_{\max}^2N^2\log(n_{\max})+2^{n_{\max}N^{
\min}})$ \\ \hline
\end{tabular}
}
\end{center}
\caption{Computational complexity analysis.}
\label{TableComplex}
\end{table}

\begin{figure}[tbp]
\begin{center}
\includegraphics[scale=0.24]{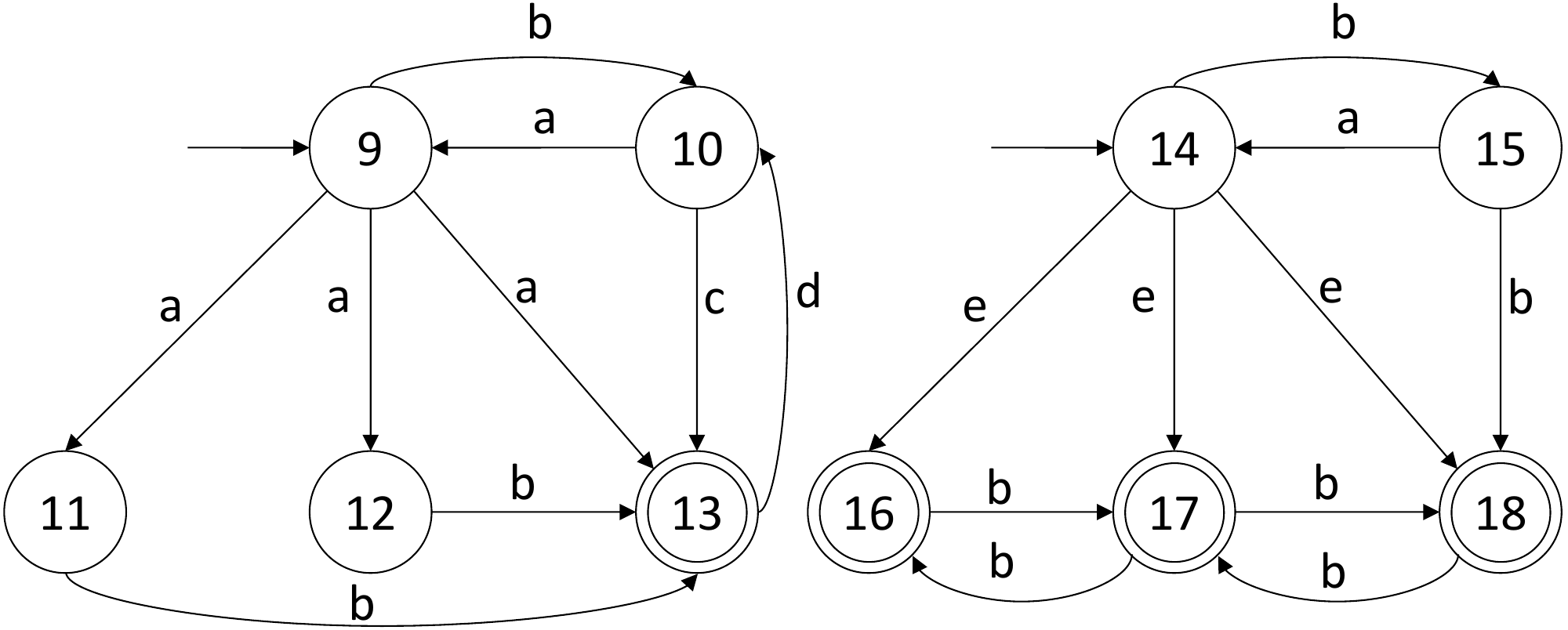}
\end{center}
\caption{FSM $M_3$ in the left and FSM $M_4$ in the right.}
\label{fig2}
\end{figure}

\section{An application to the analysis of a biological network}
\label{sec6} 
The advantages of our approach to model reduction and critical observability analysis have been illustrated in \cite{ThesisPezzuti} for Air Traffic Management systems. To show the applicability of our method in a different context we now consider a biological system. 
Systems biology is the mechanism by which macromolecules produce the functional properties of living cells through dynamic interactions \cite{alberghina}. The understanding of such complex biological systems requires the integration of experimental and computational research \cite{kitano}. \\
Here we propose a network of FSMs to build up a simplified model underlining the transcription network involved in bacterium \textit{Escherichia coli} sensing different nutrient conditions. In particular, we consider the system for the transport of secondary sugar galactose. 
We model the system by following \cite{Semsey} and the supplementary material therein. 
We start by defining an FSM for each gene product involved in the galactose network. According to a widespread simplifying assumption \cite{Alon}, we do not distinguish between transcripts (i.e. mRNA) and proteins: hence, we consider only the proteins and their interactions. The first FSM, depicted in Figure \ref{reg}(a), models the cAMP Receptor Protein (CRP). Starting from the initial state $0$ of inactivity, the presence of chemical signal $\cAMP$, representing the absence of glucose 6-phosphate (the preferred sugar for \textit{E. coli}), forces CRP protein to jump to a state of full activity, represented as state $1$ in the FSM . When in state $1$, the cAMP-CRP complex influences the transcription of other proteins; label $\mathrm{c}$ in the FSM serves to this purpose.
The presence of glucose 6-phosphate (represented as the logical negation $\overline{\cAMP}$ of label $\cAMP$) decreases the cAMP-CRP level bringing back the protein at state $0$ of inactivity. 
FSM associated with GalR is depicted in Figure \ref{reg}(b) and works in a similar way; the difference is that this protein represents the intracellular presence of D-galactose sugar. Thus it jumps to a state of full activity in the presence of label $\Dgal$.  
Figure \ref{transport} depicts FSMs modeling the two proteins GalP and MglB involved in the permease and transport of the galactose sugar. 
We start by describing the FSM modeling GalP. The full expression of GalP (state $3$) requires the full activity of both proteins in Figure \ref{reg}. Hence, from a modeling viewpoint, state $3$ is reached after both events $\mathrm{c}$ and $\mathrm{g}$, in any order, occur. (Events $\mathrm{\overline{c}}$ and $\mathrm{\overline{g}}$ are interpreted as the logical negations of events $\mathrm{c}$ and $\mathrm{g}$, respectively). The internal concentration of glucose 6-phosphate is assumed to be not known which implies that the set of initial states is given by $\{0,1\}$. Since we are interested in the full activity of the proteins involved in the transport of galactose sugar (state $3$ in Figure \ref{transport}(a)), we consider as critical, the states representing low or partial activity, i.e. states $0$, $1$ and $2$. FSM modeling the MglB system works similarly.
\begin{figure}[ht!]
\centering
\subfigure[FSM $\CRP$.]{
\begin{tikzpicture}[->,shorten >=0.1pt,node distance=2.4cm,auto,state/.style={state without output,thin,font=\normalsize},initial text=,bend angle=20, inner sep=2pt ]
\node[state, initial, initial where= left] (0) {0};
\node[state] (1) [right of=0] {1};
\path[->,thin, every node/.style={font=\small}] 
(0) edge [bend left] node {$\cAMP$} (1)
    edge [loop above] node {$\mathrm{\overline{c}}$} ()
(1) edge [bend left] node {$\overline{\cAMP}$} (0)
    edge [loop above] node {$\mathrm{c}$} ();
\end{tikzpicture}
}
\subfigure[FSM $\GalR$.]{
\begin{tikzpicture}[->,shorten >=0.1pt,node distance=2.4cm,auto,state/.style={state without output,thin,font=\normalsize},initial text=,bend angle=20, inner sep=2pt]
\node[state, initial] (0) {0};
\node[state] (1) [right of=0] {1};
\path[->,thin, every node/.style={font=\small}] 
(0) edge [bend left] node {$\Dgal$} (1)
    edge [loop above] node {$\mathrm{\overline{g}}$} ()
(1) edge [bend left] node {$\overline{\Dgal}$} (0)
    edge [loop above] node {$\mathrm{g}$} ();
\end{tikzpicture}
}
\caption{Regulation of galactose}
\label{reg}
\end{figure}
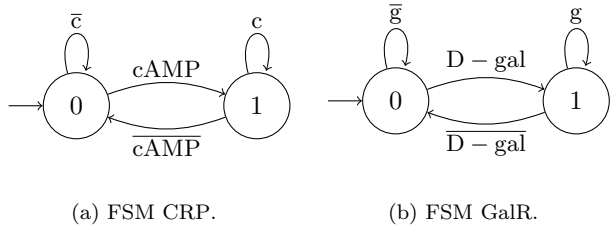 
\begin{figure}[ht!]
\centering
\subfigure[FSM $\GalP$.]{
\begin{tikzpicture}[->,shorten >=0.1pt,node distance=2.3cm,auto,state/.style={state without output,thin,font=\normalsize},initial text=,bend angle=30, inner sep=2pt]
\node[state, initial,initial where= below, accepting] (0) {0};
\node[state, initial,initial where= below, accepting] (1) [right of=0] {1};
\node[state, accepting] (2) [below of=0] {2};
\node[state] (3) [right of=2] {3};
\path[->,thin, every node/.style={font=\normalsize}] 
(0) edge [bend left] node {$\mathrm{c}$} (1)
    edge [loop above] node {$\mathrm{\overline{c}},\mathrm{\overline{g}}$} ()
(1) edge [bend left] node [] {$\mathrm{\overline{c}}$} (0)
    edge [loop above] node {$\mathrm{c},\mathrm{\overline{g}}$} ()
(0) edge [bend left] node [] {$\mathrm{g}$} (2)
(2) edge [bend left] node [] {$\mathrm{\overline{g}}$} (0)
    edge [loop below] node {$\mathrm{\overline{c}},\mathrm{g}$} ()
(1) edge [bend left] node [] {$\mathrm{g}$} (3)
(3) edge [bend left] node [] {$\mathrm{\overline{g}}$} (1)
    edge [loop below] node {$\mathrm{c},\mathrm{g}$} ()
(2) edge [bend left] node [] {$\mathrm{c}$} (3)
(3) edge [bend left] node {$\mathrm{\overline{c}}$} (2);
\end{tikzpicture}
}
\subfigure[FSM $\MglB$.]{
\begin{tikzpicture}[->,shorten >=0.1pt,node distance=2.3cm,auto,state/.style={state without output,thin,font=\normalsize},initial text=,bend angle=30, inner sep=2pt]
\node[state, initial, initial where= below, accepting] (0) {0};
\node[state, initial, initial where= below, accepting] (1) [right of=0] {1};
\node[state, accepting] (2) [below of=0] {2};
\node[state] (3) [right of=2] {3};
\path[->,thin, every node/.style={font=\normalsize}] 
(0) edge [bend left] node {$\mathrm{c}$} (1)
    edge [loop above] node {$\mathrm{\overline{c}},\mathrm{\overline{g}}$} ()
(1) edge [bend left] node [] {$\mathrm{\overline{c}}$} (0)
    edge [loop above] node {$\mathrm{c},\mathrm{\overline{g}}$} ()
(0) edge [bend left] node [] {$\mathrm{g}$} (2)
(2) edge [bend left] node [] {$\mathrm{\overline{g}}$} (0)
    edge [loop below] node {$\mathrm{\overline{c}},\mathrm{g}$} ()
(1) edge [bend left] node [] {$\mathrm{g}$} (3)
(3) edge [bend left] node [] {$\mathrm{\overline{g}}$} (1)
    edge [loop below] node {$\mathrm{c},\mathrm{g}$} ()
(2) edge [bend left] node [] {$\mathrm{c}$} (3)
(3) edge [bend left] node {$\mathrm{\overline{c}}$} (2);
\end{tikzpicture}
}
\caption{Permease and transport of galactose}
\label{transport}
\end{figure}
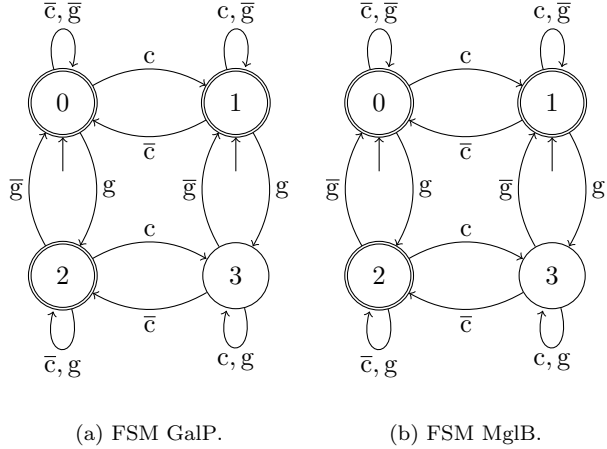
Detection of full activity of the proteins involved on the basis of available information can be rephrased as a critical observability problem. 
Consider the whole network of FSMs $\mathcal{N}:=\{\CRP,\GalR,\GalP,\MglB\}$. 
The first step of Algorithm \ref{alg3} consists in computing the quotient network $\mathcal{N}^{\min}$. Since GalP and MglB are isomorphic (hence, bisimilar), the equivalence classes induced by the bisimulation equivalence on $\mathcal{N}$ are $\mathcal{E}_1:=\{\CRP\},\mathcal{E}_2:=\{\GalR\},\mathcal{E}_3:=\{\GalP,\MglB\}$. The resulting network $\mathcal{N}^{\min}$ can be chosen as $\{\CRP,\GalR,\GalP\}$. By applying Algorithm \ref{alg} we get that $\mathcal{N}^{\min}$ is not critically observable. This is also evident from the fact that FSM $\GalP$ is not critically observable (with stimulus $\mathrm{g}$ there is a transition from the initial state $0$ to the critical state $2$ and a transition from the initial state $1$ to the non critical state $3$) and this situation lasts even after the composition of the whole network.
From the systems biology perspective, this means that we cannot know the activity level of proteins needed for the transport of galactose inside the bacterium without any further information, which could instead be obtained by additional measurements at the expense of higher costs of the experimental set-up. 
The model proposed in this section is simplified and hence, it does not capture the complexity of the whole network, which in fact is out of the scope of the present paper. In future work we plan to extend the proposed model along two directions: (i) The proposed network of FSMs provides a description of the proteins at steady state (in fact the work in \cite{Semsey}, inspiring our model, makes use of boolean networks). In future work we plan to extend the description to the transient regime which is of great interest in the systems biology community and which also is possible with the use of FSMs formalism (and not with boolean networks as used in \cite{Semsey}); (ii) A second research future direction consists in considering a larger and possibly more realistic transcription network. The outcome of the analysis would assist scientists in better understanding interaction mechanism in large-scale complex transcription networks.

\section{Discussion}

\label{sec7}
In this paper, we proposed decentralized critical observers for networks of
FSMs. On--line detection of critical states is performed by
local critical observers, each one associated with an FSM of the network.
For the design of local observers, efficient algorithms  were provided which are based on
on--the-fly techniques. Model reduction of networks of FSMs via 
a notion of bisimulation equivalence that takes into account criticalities 
was shown to facilitate the design of distributed observers for
the original network. \\
A discussion on connections with observability\footnote{Here, the term "observability" is used in a broader sense.} notions available in the current literature follows. 
These notions can be roughly categorized along the following directions: \textit{language} (L) vs. \textit{state space} (S) based notions; \textit{centralized} (C) vs. \textit{decentralized} (D) architectures; notions employed for purely \textit{analysis} purposes (A) vs. notions instrumental to address \textit{control design} (CD). An exhaustive review of existing literature in this regard is out of the scope of the present paper. 
We only recall here:  
within (L)--(C)--(A) the notions of language observability \cite{Wonham:88} and of diagnosability \cite{Sampath:95}; 
within (L)--(C)--(CD) optimal sensor activation in \cite{Wang:10,Wang:Aut10}; 
within (L)--(D)--(A) the notion of co--diagnosability in e.g. \cite{Zhou08,Lafortune2006,Debouk02,Schmidt13,Su05}, extending the one of diagnosability to a decentralized setting; 
within (L)--(D)--(CD) the notion of co--observability \cite{Rudie:89} extending the one of language observability to a decentralized setting, 
 and the work \cite{Wang:11} proposing an extension of the notion of co--observability to dynamic observations in controlled DESs and results for translating co--observability problems into co--diagnosability problems;  
within (S)--(C)--(A) the notions of diagnosability with a state space approach in \cite{Wonham:03}, current state observability \cite{BalluchiHSCC02}, opacity \cite{opacity}, detectability e.g. \cite{Detectability:07} and 
efficient algorithms for the computation of indistinguishable states in e.g. \cite{Wang:07}; 
within (S)--(D)--(A) the concept of states disambiguation in e.g. \cite{Rudie:03} and the notion of critical observability in \cite{IJRNC08} and in the present paper.\\
Many papers mentioned above consider deterministic FSMs with partially observable transitions while in our paper we consider nondeterministic FSMs with observable transitions; however, it is well known how to translate the first class of FSMs into the latter and vice versa. Moreover, the discussion in \cite{Wang:07} shows that the problems of guaranteeing observability in the sense of \cite{Wonham:88,Rudie:89}, can be transformed into problems of states disambiguation. \\
Connections with the concept of states disambiguation in e.g. \cite{Rudie:03} which is in (S)--(D)--(A) as our paper, follow. 
The work \cite{Rudie:03} deals with states disambiguation of DESs communicating each other: given a pair of DESs, the problem addressed consists in finding minimal information needed to be shared between two agents so that each agent is able to distinguish between the states of its DES. 
States disambiguation deals with studying conditions under which any pair of states of a DES can be distinguished on the basis of the traces associated to state runs of the DES and ending in the two states. 
It is readily seen that an FSM where \textit{all} states are disambiguated is also critically observable. The converse implication is not true in general: for example the FSM $M_1$ in Remark \ref{remark1} is critically observable but states $1$ and $2$ are not disambiguated. 
In fact, the notion of critical observability can be viewed as an extension of the concept of \textit{states disambiguation} to \textit{sets disambiguation}, where states in the set of critical states $C$ of an FSM $M=(X,X^0,\Sigma,\delta,C)$ need to be disambiguated from states in the set $X\backslash C$  obtained as the complement of $C$ on $X$. \\
Apart from technical differences between the concepts of states disambiguation of e.g. \cite{Rudie:03} and critical observability, the present paper approaches the efficient design of decentralized critical observers by extending techniques based on formal
methods. To the best of our knowledge, the proposed approach was not explored before in the literature on discrete event systems with the only exception of \cite{Wonham:03}. \\
A discussion on connections with \cite{Wonham:03} follows. While our work is within (S)--(D)--(A), the work \cite{Wonham:03} is within (S)--(C)--(A): it proposes a state space approach to the study of diagnosability of single FSMs and model reduction via bisimulation, as a tool to facilitate the check of the fault diagnosability property. Regarding the notions employed, while critical observability requires the immediate detection of a critical state, diagnosability notions as also the one considered in 
\cite{Wonham:03} allow for a finite delay before fault detection; moreover, while critical states are needed to be detected whenever they are reached, faults are detected only the first time they are reached. Regarding model reduction schemes employed, while \cite{Wonham:03} performs bisimulation--based reduction of \textit{single} FSMs, our approach proposes a bisimulation--based reduction of \textit{networks} of FSMs such that the FSMs composing the network are never composed; this approach allows model reduction at a higher level of abstraction, and therefore is more effective, as also substantiated by the computational complexity analysis performed and Example \ref{example1tot} presented. \\
In future work we plan to extend the formal methods techniques proposed in this paper, from decentralized critical observability to co--diagnosability. Useful insights in this regard are reported in \cite{Wonham:03}.

\textit{Acknowledgements: } We wish to thank Sina Lessanibahri for participating in fruitful discussions at the beginning of this project. 
We also thank Pasquale Palumbo for fruitful discussions on the systems biology example in Section \ref{sec6}.

\bibliographystyle{plain}
\bibliography{MAREAbib,biblio1,biblio_gal}

\end{document}